\documentclass[reqno]{amsart}
\usepackage{enumerate}
\usepackage{amsmath}%
\usepackage{amsfonts}%
\usepackage{amssymb}%
\usepackage{graphicx}
\usepackage{mathrsfs}
\usepackage{hyperref}

\newtheorem{theorem}{Theorem}
\theoremstyle{plain}

\newtheorem{claim}[theorem]{Claim}

\newtheorem{corollary}[theorem]{Corollary}

\newtheorem{definition}[theorem]{Definition}

\newtheorem{fact}[theorem]{Fact}
\newtheorem{lemma}[theorem]{Lemma}

\newtheorem{proposition}[theorem]{Proposition}
\newtheorem{remark}[theorem]{Remark}

\numberwithin{equation}{section}
\numberwithin{theorem}{section}
\numberwithin{case}{section}

\numberwithin{subcase}{case}

\def\ba{\backslash}
\def\A{\mathcal{A}}

\def\D{\mathcal{D}}

\def\F{\mathcal{F}}
\def\G{\mathcal{G}}
\def\h{\mathcal{H}}
\def\K{\mathcal{K}}

\def\M{\mathcal{M}}

\def\T{\mathcal{T}}
\def\Q{\mathcal{Q}}
\def\Y{\mathcal{Y}}
\def \a{\alpha}
\def \e{\epsilon}

\def \r{\gamma}

\begin{document}
\title[vertex degree threshold for loose Hamilton cycles in 3-graphs]{Minimum vertex degree threshold for loose Hamilton cycles in 3-uniform hypergraphs}
\thanks{The second author is partially supported by NSA grant H98230-12-1-0283 and NSF grant DMS-1400073.}
\author{Jie Han}
\address[Jie Han and Yi Zhao]
{Department of Mathematics and Statistics,\newline
\indent Georgia State University, Atlanta, GA 30303}
\email[Jie Han]{jhan22@gsu.edu}
\author{Yi Zhao}
\email[Yi Zhao]{\texttt{yzhao6@gsu.edu}}%
\date{\today}
\begin{abstract}
We show that for sufficiently large $n$, every 3-uniform hypergraph on $n$ vertices with minimum vertex degree at least $\binom{n-1}2 - \binom{\lfloor\frac34 n\rfloor}2 + c$, where $c=2$ if $n\in 4\mathbb N$ and $c=1$ if $n\in 2\mathbb{N}\setminus 4\mathbb N$, contains a loose Hamilton cycle. This degree condition is best possible and improves on the work of Bu\ss, H\`an and Schacht who proved the corresponding asymptotical result.
\end{abstract}

\maketitle

\section{Introduction}
The study of Hamilton cycles is an important topic in graph theory. In recent years, researchers have worked on extending the classical theorem of Dirac \cite{dirac} on Hamilton cycles to hypergraphs -- see recent surveys of \cite{RR, KuOs14ICM}.

Given $k\ge 2$, a $k$-uniform hypergraph (in short, \emph{$k$-graph}) consists of a vertex set $V$ and an edge set $E\subseteq \binom{V}{k}$, where every edge is a $k$-element subset of $V$. For $1\le l< k$, a $k$-graph is called an \emph{$l$-cycle} if its vertices can be ordered cyclically such that each of its edges consists of $k$ consecutive vertices and every two consecutive edges (in the natural order of the edges) share exactly $l$ vertices. 
(If we allow $l=0$, then a $0$-cycle is merely a matching and perfect matchings have been intensively studied recently, e.g. \cite{AFHRRS, CzKa, HPS, Khan1, Khan2, KOT, RRS09, TrZh12, TrZh13})
In $k$-graphs, a $(k-1)$-cycle is often called a \emph{tight} cycle while a $1$-cycle is often called a \emph{loose} cycle. We say that a $k$-graph contains a \emph{Hamilton $l$-cycle} if it contains an $l$-cycle as a spanning subhypergraph. Note that a Hamilton $l$-cycle of a $k$-graph on $n$ vertices contains exactly $n/(k - l)$ edges, implying that $k- l$ divides $n$.

Given a $k$-graph $\h$ with a set $S$ of $d$ vertices (where $1 \le d \le k-1$) we define $\deg_{\h} (S)$ to be the number of edges containing $S$ (the subscript $\h$ is omitted if it is clear from the context). The \emph{minimum $d$-degree $\delta _{d} (\h)$} of $\h$ is the minimum of $\deg_{\h} (S)$ over all $d$-vertex sets $S$ in $\h$.  We refer to  $\delta _1 (\h)$ as the \emph{minimum vertex degree} and  $\delta _{k-1} (\h)$ the \emph{minimum codegree} of $\h$.

\subsection{Hamilton cycles in hypergraphs}
Confirming a conjecture of Katona and Kierstead \cite{KK}, R\"odl, Ruci\'nski and Szemer\'edi \cite{RRS06, RRS08} showed that for any fixed $k$, every $k$-graph $\h$ on $n$ vertices with $\delta_{k-1}(\h)\ge n/2 + o(n)$ contains a tight Hamilton cycle. This is best possible up to the $o(n)$ term. With long and involved arguments, the same authors \cite{RRS11} improved this to an exact result for $k=3$.
Loose Hamilton cycles were first studied by K\"uhn and Osthus \cite{KO}, who proved that
every 3-graph on $n$ vertices with $\delta_2(\h)\ge n/4 + o(n)$ contains a loose Hamilton cycle. 
Czygrinow and Molla \cite{CzMo} recently improved this to an exact result. 
The result of K\"uhn and Osthus \cite{KO} was generalized for arbitrary $k$ and arbitrary $l< k/2$ by H\`an and Schacht \cite{HS}, and independently by Keevash et al. \cite{KKMO} for $l=1$ and arbitrary $k$. Later K\"uhn, Mycroft, and Osthus \cite{KMO} obtained an asymptotically sharp bound on codegree for Hamilton $l$-cycles for all $l<k$. Hence, the problem of finding Hamilton $l$-cycles in $k$-graphs with large codegree is asymptotically solved.

Much less is known under other degree conditions. Recently R\"odl and Ruci\'nski \cite{RoRu14} gave a sufficient vertex degree condition that guarantees a tight Hamilton cycle in $3$-graphs. Glebov, Person and Weps \cite{GPW} gave a nontrivial vertex degree condition for tight Hamilton cycles in $k$-graphs for general $k$. Neither of these results is 
best possible -- see more discussion in Section~\ref{sec:CR}. 

Recently Bu\ss, H\`an, and Schacht \cite{BHS} studied the minimum vertex degree that guarantees a loose Hamilton cycle in 3-graphs  and obtained the following result.

\begin{theorem}\cite[Theorem 3]{BHS}\label{thmBHS}
For all $\gamma>0$ there exists an integer $n_0$ such that the following holds. Suppose $\h$ is a 3-graph on $n>n_0$ with $n\in 2\mathbb{N}$ and
\[
\delta_1(\h)>\left(\frac7{16}+\gamma \right)\binom n2.
\]
Then $\h$ contains a loose Hamilton cycle.
\end{theorem}

In this paper we improve Theorem~\ref{thmBHS} as follows.

\begin{theorem}[Main Result] \label{thmmain}
There exists an $n_{\ref{thmmain}}\in \mathbb N$ such that the following holds. Suppose that $\h$ is a 3-graph on $n>n_{\ref{thmmain}}$ with $n\in 2\mathbb{N}$ and
\begin{equation}
\delta_1(\h)\ge \binom{n-1}2 - \binom{\lfloor\frac34 n\rfloor}2 + c,\label{eqdeg}
\end{equation}
where $c=2$ if $n\in 4\mathbb N$ and $c=1$ otherwise.
Then $\h$ contains a loose Hamilton cycle.
\end{theorem}

The following construction shows that Theorem \ref{thmmain} is best possible. It is slightly stronger than \cite[Fact 4]{BHS}.

\begin{proposition}\label{factcounter}
For every $n\in 2\mathbb{N}$ there exists a 3-graph on $n$ vertices with minimum vertex degree $\binom{n-1}2 - \binom{\lfloor\frac34 n\rfloor}2 + c - 1$, where $c$ is defined as in Theorem~\ref{thmmain},
and which contains no loose Hamilton cycle.
\end{proposition}

\begin{proof}
Let $\h_1=(V_1,E_1)$ be the 3-graph on $n\in 2\mathbb N\setminus 4\mathbb N$ vertices such that $V_1=A\dot\cup B$
\footnote{Throughout the paper, we write $A\dot\cup B$ for $A\cup B$ when sets $A$, $B$ are disjoint.}
with $|A| =  \lceil\frac n4 \rceil- 1$ and $|B|=\lfloor \frac {3n}4 \rfloor+1$, and $E_1$ consists of all triples intersecting $A$. Note that $\delta_1(\h_1)= \binom{n-1}2 - \binom{\lfloor \frac {3n}4 \rfloor}2$. Suppose that $\h_1$ contains a loose Hamilton cycle $C$. There are $n/2$ edges in $C$ and every vertex in $A$ is contained in at most two edges in $C$.
Since $2|A|= \frac{n-2}2$, there is at least one edge of $C$ whose vertices are completely from $B$. This is a contradiction since $B$ is independent. So $\h_1$ contains no loose Hamilton cycle.

Let $\h_2 = (V_2, E_2)$ be a 3-graph on $n\in 4\mathbb{N}$ vertices such that $V_2 = A \dot\cup B$ with $|A| =  \frac n4 - 1$ and $|B| = {\frac34 n} + 1$, and $E_2$ consists of all triples intersecting $A$ and those containing both $b_1$ and $b_2$, where $b_1, b_2$ are two fixed vertices in $B$. Then $\delta_1(\h_2) = \binom{n-1}2 - \binom{\frac34 n}2 + 1$. Suppose that $\h_2$ contains a loose Hamilton cycle $C$. There are $n/2$ edges in $C$ and every vertex in $A$ is contained in at most two edges in $C$. Thus, there are at least two edges of $C$ whose vertices are completely from $B$. But due to the construction, every two edges in $B$ share two vertices so they cannot both appear in one loose cycle. This contradiction shows that $\h_2$ contains no loose Hamilton cycle.
\end{proof}

As a typical approach of obtaining exact results, we distinguish the \emph{extremal} case from the \emph{nonextremal} case and solve them separately.

\begin{definition}
Given $\beta>0$, a 3-graph $\h$ on $n$ vertices is called $\beta$-extremal if there is a set $B\subseteq V(\h)$, such that $|B| = \lfloor3n/4\rfloor$ and $e(B)\le \beta n^3$.
\end{definition}

\begin{theorem}[Extremal Case]\label {lemE}
There exist $\beta>0$ and $n_{\ref{lemE}}\in \mathbb N$ such that the following holds. Let $n>n_{\ref{lemE}}$ be an even integer. Suppose that $\h$ is a 3-graph on $n$ vertices satisfying \eqref{eqdeg}. If $\h$ is $\beta$-extremal, then $\h$ contains a loose Hamilton cycle.
\end{theorem}

\begin{theorem}[Nonextremal Case]\label {lemNE}
For any $\beta>0$, there exist $\r>0$ and $n_{\ref{lemNE}}\in \mathbb N$ such that the following holds.
Let $n>n_{\ref{lemNE}}$ be an even integer. Suppose that $\h$ is a 3-graph on $n$ vertices satisfying
$\delta_1(\h)\ge \left(\frac7{16} - \r \right)\binom n {2}$. If $\h$ is not $\beta$-extremal, then $\h$ contains a loose Hamilton cycle.
\end{theorem}

Theorem \ref{thmmain} follows from Theorems \ref{lemE} and \ref{lemNE} immediately by choosing $\beta$ from Theorem~\ref{lemE} and letting $n_{\ref{thmmain}}=\max\{n_{\ref{lemE}}, n_{\ref{lemNE}}\}$.

Let us discuss our proof ideas here. The proof of Theorem~\ref{lemE} is somewhat standard (though non-trivial). The proof of Theoreom~\ref{lemNE} follows the approach in the previous work \cite{BHS, HS, KMO, RRS06, RRS08, RRS11}. Roughly speaking, we use the \emph{absorbing method} initiated by R\"odl, Ruci\'nski and Szemer\'edi, which reduces the task of finding a loose Hamilton cycle to finding constantly many vertex-disjoint loose paths that covers almost all the vertices of the 3-graph.
More precisely, we first apply the Absorbing Lemma (Lemma~\ref{lemA}) and obtain a (short) absorbing path $\mathcal P_0$ which can absorb any smaller proportion of vertices. Second, apply the Reservoir Lemma (Lemma~\ref{lemR}) and find a small reservoir set $R$ whose vertices may be used to connect any constant number of loose paths to a loose cycle. Third, apply the Path-tiling Lemma (Lemma~\ref{lemP}) in the remaining 3-graph and obtain constantly many vertex-disjoint loose paths covering almost all the vertices. Fourth, connect these paths (including $\mathcal P_0$) together by the reservoir $R$ and get a loose cycle $C$. Finally we absorb the vertices not in $V(C)$ to $\mathcal P_0$ and obtain the desired loose Hamilton cycle.

The Absorbing Lemma and the Reservoir Lemma are not very difficult and already proven in \cite{BHS}. Thus the main step is to prove the Path-tiling Lemma, under the assumption $\delta_1(\h)\ge \left(\frac7{16} - \r \right)\binom n {2}$ and that $\h$ is not $\beta$-extremal (in contrast, $\delta_1(\h)\ge (\frac7{16}+\r)\binom n2$ is assumed in \cite{BHS}). As shown in \cite{BHS, HS}, after applying the (weak) Regularity Lemma, it suffices to prove that the cluster 3-graph $\K$ can be tiled almost perfectly by some particular 3-graph.
For example, the 3-graph $\M$ given in \cite{BHS} has the vertex set $[8] = \{1, 2, \dots, 8\}$ and edges $123, 345, 456, 678$.\footnote{Throughout the paper, we often represent a set $\{v_1,v_2,\dots,v_k\}$ as $v_1v_2\cdots v_k$.} Since it is hard to find an $\M$-tiling directly, the authors of \cite{BHS} found a \emph{fractional} $\M$-tiling instead and converted it to an (integer) $\M$-tiling by applying the Regularity Lemma again. In this paper we consider a much simpler 3-graph $\Y$ with vertex set $[4]$ and edges $123, 234$, and obtain an almost perfect $\Y$-tiling in $\K$ directly. Interestingly, $\Y$-tiling
was studied (via the codegree condition) in the very first paper on loose Hamilton cycles \cite{KO}.

Comparing with the first exact result on Hamilton cycles in hypergraphs \cite{RRS11}, our proof is much shorter because the Absorbing and Reservoir Lemmas in \cite{RRS11} are much harder to prove.

The rest of the paper is organized as follows: we prove Theorem~\ref{lemNE} in Section 2 and Theorem~\ref{lemE} in Section 3, and give concluding remarks in Section 4.

\subsection{Notations}

Given a vertex $v$ and disjoint vertex sets $S, T$ in a 3-graph $\h$, we denote by $\deg_{\h}(v, S)$ the number of edges that contain $v$ and two vertices from $S$, and by $\deg_{\h}(v, ST)$ the number of edges that contain $v$, one vertex from $S$ and one vertex from $T$.
Furthermore, let $\overline \deg_\h (v, S)=\binom{|S|}2-\deg_\h (v,S)$ and $\overline \deg_\h (v, ST)=|S|\cdot |T|-\deg_\h (v, ST)$.
Given not necessarily disjoint sets $X, Y, Z\subseteq V(\h)$, we define
\begin{align*}
&E_{\h}(XYZ)=\{xyz\in E(\h): x\in X, y\in Y, z\in Z\}, \\
&\overline E_{\h}(XYZ)=\left\{xyz\in \binom{V(\h)}{3}\setminus E(\h): x\in X, y\in Y, z\in Z \right\},
\end{align*}
$e_{\h}(XYZ)=|E_{\h}(XYZ)|$, and $\overline e_{\h}(XYZ)=|\overline E_{\h}(XYZ)|$. The subscript $\h$ is often omitted when it is clear from the context.

A \emph{loose} path $\mathcal P=v_1v_2$ $\cdots$ $v_{2k+1}$ is a 3-graph on $\{v_1, v_2, \dots, v_{2k+1}\}$ with edges $v_{2i-1}v_{2i}v_{2i+1}$ for all $i\in [k]$. The vertices $v_1$ and $v_{2k+1}$ are called the \emph{ends} of $\mathcal P$.

\section{Proof of Theorem \ref{lemNE}}

In this section we prove Theorem \ref{lemNE} by following the same approach as in \cite{BHS}.

\subsection{Auxiliary lemmas and Proof of Theorem~\ref{lemNE}}

For convenience, we rephrase the Absorbing Lemma \cite[Lemma 7]{BHS} as follows.\footnote{Lemma 7 in \cite{BHS} assumes that $\delta_1(\h) \ge (\frac{5}{8} + \r)^2 \binom n2$ and returns $|V(\mathcal P)|\le \r^7 n$ with $|U|\le \frac{\r^{14}}{14336} n$. We simply take their $\r^7$ as our $\r_1$ and thus $\r_1 \le \left(\sqrt{\frac{13}{32}} - \frac{5}{8}\right)^7 \approx 10^{-14}$.}

\begin{lemma}[Absorbing Lemma]\label{lemA}
For any $0<\r_1 \le 10^{-14}$ there exists an integer $n_{\ref{lemA}}$ such that the following holds. Let $\h$ be a 3-graph on $n>n_{\ref{lemA}}$ vertices with
$
\delta_1(\h) \ge \frac{13}{32}\binom n2.
$
Then there is a loose path $\mathcal P$ with $|V(\mathcal P)|\le \gamma_1 n$ such that for every subset $U\subseteq V\ba V(\mathcal P)$ with $|U|\le \r_1^3 n$ and $|U|\in 2\mathbb N$ there exists a loose path $\mathcal Q$ with $V(\mathcal Q) = V(\mathcal P)\cup U$ such that $\mathcal P$ and $\mathcal Q$ have the same ends.
\end{lemma}

We also need the Reservoir Lemma \cite[Lemma 6]{BHS}.
\begin{lemma}[Reservoir Lemma]\label{lemR}
For any $0< \gamma_2 < 1/4$ there exists an integer $n_{\ref{lemR}}$ such that for every 3-graph $\h$ on $n>n_{\ref{lemR}}$ vertices satisfying
\[
\delta_1(\h)\ge (1/4 + \gamma_2)\binom n2,
\]
there is a set $R$ of size at most $\gamma_2 n$ with the following property: for every  $k\le \gamma_2^3 n/12$ mutually disjoint pairs $\{a_i,b_i\}_{i\in [k]}$ of vertices from $V(\h)$ there are $3k$ vertices $u_i, v_i, w_i$, $i\in [k]$ from $R$ such that $a_iu_iv_i, v_iw_ib_i\in \h$ for all $i\in [k]$.
\end{lemma}

The main step in our proof of Theorem~\ref{lemNE} is the following lemma, which is stronger than\cite[Lemma 10]{BHS}.

\begin{lemma}[Path-tiling lemma]\label{lemP}
For any $0<\r_3, \a<1$ there exist integers $p$ and $n_{\ref{lemP}}$ such that the following holds for $n>n_{\ref{lemP}}$. Suppose $\h$ is a 3-graph on $n$ vertices with minimum vertex degree
\[
\delta_1(\h)\ge \left(\frac7{16} - \r_3 \right)\binom n {2},
\]
then there are at most $p$ vertex disjoint loose paths in $\h$ that together cover all but at most $\a n$ vertices of $\h$ unless $\h$ is $2050\r_3$-extremal.
\end{lemma}

\begin{proof}[Proof of Theorem \ref{lemNE}]
Given $\beta>0$, let $\r = \min \{\frac{\beta}{4101}, 10^{-14}\}$. We choose $n_{\ref{lemNE}} = \max\{n_{\ref{lemA}}, 2n_{\ref{lemR}}, 2n_{\ref{lemP}}, 192(p + 1)/(\r/3)^{9}\}$, where $p$ is the constant returned from Lemma \ref{lemP} with $\r_3 = 2\r$ and $\a =  (\r/3)^{3}/2$. Let $n> n_{\ref{lemNE}}$ be an even integer.

Suppose that $\h = (V, E)$ is a 3-graph on $n$ vertices with $\delta_1(\h)\ge \left(\frac7{16} - \r \right)\binom n {2}$. Since $\frac7{16} - \r >\frac{13}{32}$, we can apply Lemma~\ref{lemA} with $\r_1= \r/3$ and obtain an absorbing path $\mathcal P_0$ with ends $a_0, b_0$. We next apply Lemma~\ref{lemR} with $\r_2= (\r/3)^{3}/2$ to $\h[(V\setminus V(\mathcal P_0))\cup \{a_0, b_0\}]$ and obtain a reservoir $R$. Let $V' = V\setminus (V(\mathcal P_0)\cup R)$ and $n' = |V'|$. Note that $n - n' \le \r_1 n +\r_2 n < \r n/2$. The induced subhypergraph $\h' = \h[V']$ satisfies
\[
\delta_1(\h')\ge \left(\frac7{16} -\r\right)\binom {n} 2 -\frac{\r}{2} n \cdot (n-2) > \left(\frac7{16} - 2\r\right)\binom {n'} 2.
\]
Applying Lemma \ref{lemP} to $\h'$ with $\r_3 = 2\r$ and $\a = (\r/3)^{3}/2$, we obtain at most $p$ vertex disjoint loose paths that cover all but at most $\a n'$ vertices of $\h'$, unless $\h'$ is $2050\r_3$-extremal.
In the latter case, there exists $B'\subseteq V'$ such that $|B'|=\lfloor\frac34 n'\rfloor$ and $e(B')\le 4100\r (n')^3$. Then we add $\lfloor\frac34 n\rfloor - \lfloor\frac34 n'\rfloor < \r n/2$ arbitrary vertices from $V\setminus B'$ to $B'$ to get a vertex set $B$ such that $|B| = \lfloor\frac34 n\rfloor$ and
\[
e(B)\le 4100\r (n')^3 + \frac{\r n}{2} \binom{n-1}{2} <4101 \r n^3 \le \beta n^3,
\]
which means that $\h$ is $\beta$-extremal, a contradiction. In the former case, denote these loose paths by $\{\mathcal P_i\}_{i\in [p']}$ for some $p'\le p$, and their ends by $\{a_i, b_i\}_{i\in [p']}$. The choice of $n_{\ref{lemNE}}$ guarantees that $p'+1\le p+1\le {\r_2^3}n/{24}$. We can thus connect $\{a_i, b_{i+1}\}_{0\le i\le p'-1}\cup \{a_{p'}, b_0\}$ by using vertices from $R$ obtaining a loose cycle $C$. Since $|V\setminus C|\le |R|+\a n' \le \r_2 n + \r_2 n' \le \r_1^{3}n$, we can use $\mathcal P_0$ to absorb all unused vertices in $R$ and uncovered vertices in $V'$.
\end{proof}

The rest of this section is devoted to the proof of Lemma \ref{lemP}.

\subsection{Proof of Lemma \ref{lemP}}

Following the approach in \cite{BHS}, we will use the weak regularity lemma which is a straightforward extension of Szemer\'edi's regularity lemma for graphs \cite{Sze}. Below we only state this lemma for 3-graphs.

Let $\h = (V, E)$ be a 3-graph and let $A_1, A_2, A_3$ be mutually disjoint non-empty subsets of $V$. We define $e(A_1, A_2, A_3)$ to be the number of edges with one vertex in each $A_i$, $i\in [3]$, and the density of $\h$ with respect to ($A_1, A_2, A_3$) as
\[
d(A_1, A_2, A_3) = \frac{e(A_1, A_2, A_3)}{|A_1||A_2||A_3|}.
\]
Given $\e>0$, the triple $(V_1, V_2, V_3)$ of mutually disjoint subsets $V_1, V_2, V_3\subseteq V$ is called \emph{$\e$-regular} if
\[
|d(A_1, A_2, A_3) - d(V_1,  V_2, V_3)|\le \e
\]
for all triple of subsets of $A_i\subseteq V_i$, $i\in [3]$, satisfying $|A_i|\ge \e |V_i|$. We say ($V_1, V_2, V_3$) is \emph{($\e, d$)-regular} if it is $\e$-regular and $d(V_{1}, V_{2}, V_{3})\ge d$ for some $d\ge 0$. It is immediate from the definition that in an $(\e,d)$-regular triple ($V_1, V_2, V_3$), if $V_i'\subseteq V_i$ has size $|V_i'| \ge c|V_i|$ for some $c\ge \e$, then ($V_1', V_2', V_3'$) is $(\max\{\e/c,2\e\},d-\e)$-regular.

\begin{theorem}\cite[Theorem 14]{BHS}\label{thmReg}
For any $t_0\ge 0$ and $\e>0$, there exist $T_0$ and $n_0$ so that for every 3-graph $\h = (V, E)$ on $n>n_0$ vertices, there exists a partition $V = V_0 \dot \cup V_1 \dot \cup \cdots \dot\cup V_t$ such that
\begin{enumerate}[{\rm (i)}]
\item $t_0\le t\le T_0$,
\item $|V_1| = |V_2| = \dots = |V_t|$ and $|V_0|\le \e n$,
\item for all but at most $\e \binom t3$ sets $i_1 i_2 i_3\in \binom{[t]}{3}$, the triple $(V_{i_1}, V_{i_2}, V_{i_3})$ is $\e$-regular.
\end{enumerate}
\end{theorem}

A partition as given in Theorem \ref{thmReg} is called an \emph{$(\e,t)$-regular partition} of $\h$. For an $(\e, t)$-regular partition of $\h$ and $d\ge 0$ we refer to $\Q = (V_i)_{i\in [t]}$ as the family of {\em clusters} and define the {\em cluster hypergraph} $\K = \K(\e,d,\Q)$ with vertex set $[t]$ and $i_1 i_2 i_3\in \binom{[t]}{3}$ is an edge if and only if $(V_{i_1}, V_{i_2}, V_{i_3})$ is $(\e,d)$-regular.

The following corollary shows that the cluster hypergraph inherits the minimum degree of the original hypergraph. Its proof is the same as that of \cite[Proposition 15]{BHS} after we replace $7/16 + \gamma$ by $c$ (we thus omit the proof).

\begin{corollary}\label{prop15}
For $c>d>\e>0$ and $t_0\ge 0$ there exist $T_0$ and $n_0$ such that the following holds. Suppose $\h$ is a 3-graph on $n>n_0$ vertices which has minimum vertex degree $\delta_1(\h)\ge c\binom n2$. Then there exists an $(\e,t)$-regular partition $\Q$ with $t_0<t<T_0$ such that the cluster hypergraph $\K = \K(\e,d,\Q)$ has minimum vertex degree $\delta_1(\K)\ge (c - \e - d)\binom t2$.
\end{corollary}

In 3-graphs, a loose path is $3$-partite with partition sizes \emph{about} $m, m, 2m$ for some integer $m$. Proposition~\ref{prop25} below shows that every regular triple with partition sizes $m, m, 2m$ contains an almost spanning loose path as a subhypergraph. In contrast, \cite[Proposition 25]{BHS} (more generally \cite[Lemma 20]{HS}) shows that every regular triple with partition sizes $3m, 3m, 2m$ contains constant many vertex disjoint loose paths. The proof of Proposition~\ref{prop25} uses the standard approach of handling regularity.

\begin{proposition}\label{prop25}
Fix any $\e>0$, $d> 2\e$, and an integer $m \ge \frac d{\e(d-2\e)}$. Suppose that $V(\h)= V_1\cup V_2\cup V_3$ and $(V_1, V_2, V_3)$ is $(\e,d)$-regular with $|V_i| = m$ for $i = 1,3$ and $|V_2| = 2m$. Then there is a loose path $P$ omitting at most $8\e m/d+3$ vertices of $\h$.
\end{proposition}

\begin{proof}
We will greedily construct the loose path $P=v_1v_2\cdots v_{2k+1}$ such that $v_{2i}\in V_2$, $v_{4i+1}\in V_1$ and $v_{4i+3}\in V_3$ until $|V_i\setminus V(P)|<\frac{2\e}{d}|V_i|$ for some $i\in [3]$. For $j\in [3]$, let $U_j^0=V_j$ and $U_j^i=V_j\setminus \{ v_1,\dots, v_{2i-1} \}$ for $i\in [k]$. In addition, we require that for $i=0, \dots, k$,
\begin{equation}\label{eq:hyp}
\deg(v_{2i+1}, U_2^i U_r^i)\ge (d-\e)|U_2^i| |U_r^i|,
\end{equation}
where $r\equiv 2i-1 \mod 4$.
We proceed by induction on $i$. First we pick a vertex $v_1\in V_1$ such that $\deg(v_1, V_2 V_3)\ge (d-\e)|V_2||V_3|$ (thus \eqref{eq:hyp} holds for $i=0$). By regularity, all but at most $\e|V_1|$ vertices can be chosen as $v_1$. Suppose that we have selected $v_1, \dots, v_{2i-1}$. Without loss of generality, assume that $v_{2i-1}\in V_1$. Our goal is to choose $v_{2i}\in U_2^i, v_{2i+1}\in U_3^i$ such that
\begin{itemize}
  \item[(i)] $v_{2i-1} v_{2i} v_{2i+1}\in E(\h)$,
  \item[(ii)] $\deg(v_{2i+1}, U_1^i U_2^i)\ge (d-\e)|U_1^i| |U_2^i|$.
\end{itemize}

In fact, the induction hypothesis implies that $\deg(v_{2i-1}, U_2^{i-1} U_3^{i-1})\ge (d-\e)|U_2^{i-1}||U_3^{i-1}|$. Since $U_2^i = U_2^{i-1}\setminus \{v_{2i-2} \}$ and $U_3^i= U_3^{i-1}$, we have
\[
\deg(v_{2i-1}, U_2^{i} U_3^{i})\ge (d-\e)|U_2^{i-1}||U_3^{i-1}| - |U_3^{i-1}|=((d-\e)|U_2^{i-1}|-1)|U_3^{i-1}|.
\]
By regularity, at most $\e |V_3|$ vertices in $V_3$ does not satisfy (ii). So, at least
\begin{equation}
\label{eq:U23}
\deg(v_{2i-1}, U_2^{i} U_3^{i}) -\e |V_3|\cdot |U_2^{i-1}|\ge ((d-\e)|U_2^{i-1}|-1)|U_3^{i-1}|-\e |V_3|\cdot |U_2^{i-1}|
\end{equation}
pairs of vertices can be chosen as $v_{2i}, v_{2i+1}$. Since $|U_3^{i-1}|\ge \frac{2\e}{d}|V_3|$ and $|U_2^{i-1}|\ge \frac{2\e}{d}|V_2| \ge \frac4{d-2\e}$ (using $m \ge \frac d{\e(d-2\e)}$), the right side of \eqref{eq:U23} is at least
\[
\Big((d-\e)|U_2^{i-1}|-1\Big) \frac{2\e}{d}|V_3|-\e |V_3|\cdot |U_2^{i-1}| = \Big((d-2\e)|U_2^{i-1}| - 2\Big)\frac{\e}{d}|V_3|  >0,
\]
thus the selection of $v_{2i}, v_{2i+1}$ satisfying (i) and (ii) is guaranteed.

To calculate the number of the vertices omitted by $P=v_1v_2\cdots v_{2k+1}$, note that $|V_1\cap V(P)|= \lceil\frac{k+1}2\rceil$, $|V_2\cap V(P)|= k$, and $|V_3\cap V(P)|= \lfloor\frac{k+1}2\rfloor$. Our greedy construction of $P$ stops as soon as $|V_i\setminus V(P)|<\frac{2\e}{d}|V_i|$ for some $i\in [3]$. As $|V_1|=|V_3|= m= |V_2|/2$, one of the following three inequalities holds:
\[
m - \left\lceil\frac{k+1}2 \right\rceil < \frac{2\e}{d} m, \quad 2m - k < \frac{2\e}{d} 2m, \quad m - \left\lfloor\frac{k+1}2 \right\rfloor < \frac{2\e}{d} m.
\]
Thus we always have $m - \left\lceil\frac{k+1}2 \right\rceil < \frac{2\e}{d} m$, which implies that
$\frac{k+2}2 > \left(1- \frac{2\e}{d} \right)m$ or $k> 2 \left(1-\frac{2\e}{d}\right)m-2$.
Consequently,
\[
|V(\h)\setminus V(P)|= 4m - (2k+1) < 4m - \left(4 \left(1-\frac{2\e}{d} \right)m-4 + 1 \right)=\frac{8\e}{d} m +3. \qedhere
\]
\end{proof}

\medskip
Let $\Y$ be the 3-graph on the vertex set $[4]$ with edges $123,234$ (the unique 3-graph with four vertices and two edges). The following lemma is the main step in our proof of Lemma \ref{lemP}. In general, given two (hyper)graphs $\F$ and $\G$, an \emph{$\F$-tiling} is a sub(hyper)graph of $\G$ that consists of vertex disjoint copies of $\F$. The $\F$-tiling is \emph{perfect} if it is a spanning sub(hyper)graph of $\G$.

\begin{lemma}[$\Y$-tiling Lemma] \label{lem:Y}
For any $\r >0$, there exists an integer $n_{\ref{lem:Y}}$ such that the following holds. Suppose $\h$ is a 3-graph on $n>n_{\ref{lem:Y}}$ vertices with
\[
\delta_1(\h)\ge \left(\frac7{16} - \r \right)\binom n {2},
\]
then there is a $\Y$-tiling covering all but at most $2^{19}/\r$ vertices of $\h$ unless $\h$ is $2^{10} \r$-extremal.
\end{lemma}

Now we are ready to prove Lemma \ref{lemP} using the same approach as in \cite{BHS}.
\begin{proof}[Proof of Lemma \ref{lemP}]
Given $0<\r_3, \a<1$, let $n_{\ref{lemP}} = \max\{n_0, 4T_0/\e \}$ and $p = T_0/2$, where $T_0$ and $n_0$ are the constants returned from Corollary \ref{prop15} with $c=\frac 7{16} - \r_3$, $d= \r_3 /2$, $\e=\frac{\a d}{8+\a}$, and $t_0 = \max\{ n_{\ref{lem:Y}}, \frac{2^{20}}{\r_3\a}\}$.

Suppose that $\h$ is a 3-graph on $n>n_{\ref{lemP}}$ vertices with $\delta_1(\h) \ge (\frac 7{16} - \r_3) \binom n2$. By applying Corollary \ref{prop15} with the constants chosen above, we obtain an $(\e, t)$-regular partition $\Q$. The cluster hypergraph $\K = \K(\e, d, \Q)$ satisfies $\delta_1(\K)\ge (\frac 7{16} - 2 \r_3) \binom t2$. Let $m$ be the size of each cluster except $V_0$, then $(1-\e)\frac nt \le m\le \frac nt$. By Lemma \ref{lem:Y}, either $\K$ is $2^{10}(2\r_3)$-extremal, or there is a $\Y$-tiling $\mathscr Y$ of $\K$ that covers all but at most $2^{19}/(2\r_3)$ vertices of $\K$. In the first case, there exists a set $B\subseteq V(\K)$ such that $|B| = \lfloor \frac {3t}4 \rfloor$ and $e(B)\le 2^{11}\r_3 t^3 $. Let $B'\subseteq V(\h)$ be the union of the clusters in $B$. By regularity,
\begin{align*}
e(B')\le e(B) \cdot m^3 + \binom t3 \cdot d \cdot m^3 + \e \cdot \binom t3 \cdot m^3 + \binom {m}2 n,
\end{align*}
where the right-hand side bounds the number of edges from regular triples with high density, edges from regular triples with low density, edges from irregular triples and edges that are from at most two clusters. Since $m\le \frac nt$, $\e < d < \r_3$, and $t^{-2}< t_0^{-2}< \r_3$, we get
\[
e(B')\le 2^{11} \r_3 t^3 \left( \frac nt \right)^3 + d \binom t3 \left( \frac nt \right)^3 + \e \binom t3 \left( \frac nt \right)^3 + \binom {n/t}2 n <2049 \r_3 n^3.
\]
Note that $|B'|= \left\lfloor \frac {3t}4 \right\rfloor m \le \frac{3t}4\cdot \frac nt=\frac{3n}4$ implies that $|B'|\le \lfloor \frac {3n}4 \rfloor$. On the other hand,
\[
|B'|= \left\lfloor \frac {3t}4 \right\rfloor m\ge \left( \frac{3t}4-1 \right)(1-\e)\frac nt\ge \left( \frac{3t}4-\e t \right)\frac nt =\frac{3n}4-\e n,
\]
by adding at most $\e n$ vertices from $V\setminus B'$ to $B'$, we get a set $B''\subseteq V(\h)$ of size exactly $\lfloor 3n/4 \rfloor$, with $e(B'')\le e(B') + \e n \cdot n^2<2050 \r_3 n^3$. Hence $\h$ is $2050	 \r_3$-extremal.

In the second case, the union of the clusters covered by $\mathscr Y$ contains all but at most $\frac{2^{19}}{2\r_3} m+|V_0|\le \a n/4+\e n<3\a n/8$ vertices (here we use $t\ge \frac{2^{20}}{\r_3\a}$). We will apply Proposition \ref{prop25} to each member $\Y'\in \mathscr Y$. Suppose that $\Y'$ has the vertex set $[4]$ with edges $123, 234$. For $i\in [4]$, let $V_i$ denote the corresponding cluster in $\h$. We split $V_i$, $i=2,3$, into two disjoint sets $V_i^1$ and $V_i^2$ of equal sizes. Then the triples $(V_1, V_2^1, V_3^1)$ and $(V_4, V_2^2, V_3^2)$ are $(2\e ,d-\e)$-regular and of sizes $m, \frac m{2}, \frac {m}{2}$. Applying Proposition \ref{prop25} to these two triples with $m' = \frac m{2}$, we find a loose path in each triple covering all but at most $\frac{8(2\e)}{d-\e} m'+3= \a m+3$ vertices (here we need $\e=\frac{\a d}{8+\a}$).

Since $|\mathscr Y|\le t/4$, we obtain a path tiling that consists of at most $2 t/4 \le T_0/2 = p$ paths and covers all but at most
\[
2(\a m+3) \frac{t}4+\frac{3\a}8 n \le\frac{\a}2 n +\frac{3t}2 + \frac{3\a}8 n < \a n
\]
vertices. This completes the proof.
\end{proof}

\subsection{Proof of $\Y$-tiling Lemma (Lemma \ref{lem:Y})}
\begin{fact}\label{fact:Y}
Let $\h$ be a 3-graph on $m$ vertices which contains no copy of $\Y$, then $e(\h)\le \frac13 \binom m2$.
\end{fact}

\begin{proof}
Since there is no copy of $\Y$, then given any $u,v\in V(\h)$, we have that $\deg(u v)\le 1$, which implies $e(\h)\le \frac13\binom m2 \cdot 1=\frac13\binom m2$.
\end{proof}

\medskip
\begin{proof}[Proof of Lemma \ref{lem:Y}] Fix $\r>0$ and let $n\in \mathbb{N}$ be sufficiently large. Let $\h$ be a 3-graph on $n$ vertices that satisfies $\delta_1(\h)\ge (\frac7{16} - \r )\binom n {2}$. Fix a largest $\Y$-tiling $\mathscr Y=\{\Y_1, \dots, \Y_m\}$ and let $V_i=V(\Y_i)$ for $i\in [m]$. Let $V'=\bigcup_{i\in [m]}V_i$ and $U=V(\h)\setminus V'$. Assume that $|U|>2^{19}/\r$ -- otherwise we are done.

Our goal is to find a set $C$ of vertices in $V'$ of size at most $n/4$ that covers almost all the edges, which implies that $\h$ is extremal.

Let $\A_i$ be the set of all edges with exactly $i$ vertices in $V'$, for $i=0,1,2,3$. Note that $|\A_0| \le \frac13\binom{|U|}2$  by Fact \ref{fact:Y}. We may assume that $|U|< \frac34 n$ and consequently
\begin{equation}\label{eq:Um}
\quad m>\frac{n}{16}.
\end{equation}
Indeed, if $|U|\ge \frac34 n$, then taking $U'\subseteq U$ of size $\lfloor\frac34n\rfloor$, we get that $e(U')\le e(U)\le \frac13\binom{|U|}2\le \frac16 n^2<\r n^3$. Thus $\h$ is $\r$-extremal and we are done.

\begin{claim}\label{clm:A01}
$|\A_1|\le m\binom{|U|}2+12m|U|$.
\end{claim}

\begin{proof}
Let $D$ be the set of vertices $v\in V'$ such that $\deg(v, U)\ge 4|U|$. First observe that every $\Y_i\in \mathscr Y$ contains at most one vertex in $D$. Suppose instead, two vertices $x,y\in V_i$ are both in $D$. Since $\deg(x, U)\ge 4|U|>|U|/2$, the link graph\footnote{Given a 3-graph $\h$ with a vertex $v$, the \emph{link graph} of $v$ has the vertex set $V(\h)\setminus \{v\}$ and edge set $\{e\setminus \{v\}: e\in E(\h)\}$.} of $x$ on $U$ contains a path $u_1 u_2 u_3$ of length two. The link graph of $y$ on $U\setminus\{u_1,u_2,u_3\}$ has size at least $4|U|-3|U|>|U|/2$, so it also contains a path of length two, with vertices denoted by $u_4,u_5,u_6$. Note that $x u_1 u_2 u_3$ and $y u_4 u_5 u_6$ span two vertex disjoint copies of $\Y$. Replacing $\Y_i$ in $\mathscr Y$ with them creates a larger $\Y$-tiling, contradicting the maximality of $\mathscr Y$. So we conclude that $|D|\le m$. Consequently,
\[
|\A_1|\le |D|\cdot \binom {|U|}2 + |V'\setminus D|\cdot 4|U| \le m\binom{|U|}2+3m\cdot 4|U|\le m\binom{|U|}2+12m|U|.
\]
\end{proof}

Fix $u\in U$, $i\ne j \in [m]$, we denote by $L_{i,j}(u)$ the link graph of $u$ induced on $(V_i, V_j)$, namely the bipartite link graph of $u$ between $V_i$ and $V_j$. 
Let $\mathcal{T}_{\le 6}$ be the set of all triples $uij$, $u\in U$, $i,j\in [m]$ such that $e(L_{i,j}(u))\le 6$.
Let $\mathcal T_{7}^1$ be the set of all triples $uij$, $u\in U$, $i,j\in [m]$ such that\footnote{We could have $e(L_{i,j}(u))\ge 7$ here; but since $L_{i,j}(u)$ contains a vertex cover with one vertex from $V_i$ and one vertex from $V_j$, we must have $e(L_{i,j}(u))=7$.} 
$e(L_{i,j}(u))=7$ and $L_{i,j}(u)$ has a vertex cover of two vertices with one from $V_i$ and the other from $V_j$. 
Let $\T_{\ge 7}^2$ be the set of all triples $uij$, $u\in U$, $i,j \in [m]$ such that $e(L_{i,j}(u))\ge 7$ and $L_{i,j}(u)$ has a vertex cover of two vertices both from $V_i$ or $V_j$. Let $\T_{\ge 7}^3$ be the set of all triples $uij$, $u\in U$, $i,j \in [m]$ such that $e(L_{i,j}(u))\ge 7$ and $L_{i,j}(u)$ contains a matching of size three. 
By the K\"onig--Egervary theorem, a bipartite graph either contains a matching of size three or a vertex cover of size two. Thus $\mathcal{T}_{\le 6}$, $\mathcal T_{7}^{1}$, $\T_{\ge 7}^2$, $\T_{\ge 7}^3$ form a partition of $U\times \binom{[m]}2$.

\begin{figure}
\begin{center}
\includegraphics[height=3cm]{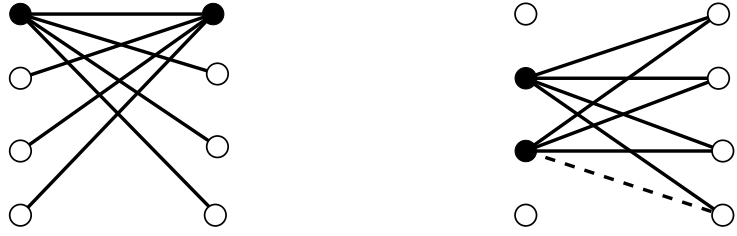}
\caption{The bipartite graph $L_{i,j}(u)$ when the triple is in $\mathcal T_7^1$ or $\mathcal T_{\ge 7}^2$, where the dotted line could be present or not.}
\end{center}
\end{figure}

In order to bound the sizes of $\mathcal{T}_{\le 6}$, $\mathcal T_{7}^{1}$, $\T_{\ge 7}^2$, $\T_{\ge 7}^3$, we need 
the following fact.

\begin{fact}\label{fact:ge7}
\begin{enumerate}[{\rm (i)}]
\item
$\h$ does not contain $i\ne j \in [m]$ and six vertices $u_1,\dots, u_6\in U$ such that $u_1,\dots, u_6$ have the same (labeled) link graph on $(V_i, V_j)$ and $u_1ij\in \T_{\ge 7}^3$.
\item
$\h$ does not contain distinct $i,j, k\in [m]$ and eight vertices  $u_1,\dots, u_8\in U$ such that the following holds. First, $u_1,\dots, u_4$ share the same link graph on $(V_i, V_j)$, and $u_5, \dots, u_8$ share the same link graph on $(V_i, V_k)$. Second, $u_1ij\in \T_{\ge 7}^2$ with the vertex cover in $V_j$ and $u_5 ik\in \T_{\ge 7}^2$ with the vertex cover in $V_k$.
\end{enumerate}
\end{fact}

\begin{proof}
To see Part (i), since there is a matching of size three in the (same) link graph of $u_1,\dots, u_6$, say, $a_1b_1, a_2b_2, a_3b_3$, then $u_1u_2a_1b_1$, $u_3u_4a_2b_2$ and $u_5u_6a_3b_3$ span three copies of $\Y$. Replacing $\Y_i, \Y_j$ by them gives a $\Y$-tiling larger than $\mathscr Y$, a contradiction.

To see Part (ii), assume that $V_i=\{a,b,c,d\}$. Suppose that the vertex cover of $L_{i,j}(u_1)$ is $\{x_1,y_1\}\subseteq V_j$ and the vertex cover of $L_{i,k}(u_5)$ is $\{x_2,y_2\}\subseteq V_k$. Since $u_1ij\in \T_{\ge 7}^2$, at most one pair from $\{x_1,y_1\}\times \{a,b\}$ is not in $L_{i,j}(u_1)$. Analogously at most one pair from $\{x_2,y_2\}\times \{c,d\}$ is not in $L_{i,k}(u_5)$. Thus, without loss of generality, we may assume that $x_1a, y_1b\in L_{i,j}(u_1)$ and $x_2c, y_2d\in L_{i,k}(u_5)$. Since $u_1, \dots, u_4$ share the same link graph on $(V_i, V_j)$, $u_1u_2x_1a$, $u_3u_4y_1b$ span two copies of $\Y$. Similarly, $u_5u_6x_2c$ and $u_7u_8y_2d$ span two copies of $\Y$. Replacing $\Y_i, \Y_j, \Y_k$ by these four copies of $\Y$ gives a $\Y$-tiling larger than $\mathscr Y$, a contradiction.
\end{proof}

We now show that most triples $uij$, $u\in U$, $i,j\in [m]$ are in $\mathcal T_{7}^{1}$.
\begin{claim}\label{clm:T}
\begin{enumerate}[{\rm (i)}]
\item $|\T_{\ge 7}^3|\le \binom m2 \cdot 2^{16}\cdot 5$,
\item $|\T_{\ge 7}^2|\le 756\binom m2 + m\cdot |U|$, 
\item $|\T_{\le 6}|\le \r \binom n2|U|+ 2^{22}\binom m2$,
\item $|\mathcal T_{7}^{1}|\ge \binom{m}2|U| -  \r n^2|U|$.
\end{enumerate}
\end{claim}

\begin{proof}
To see Part (i), by Fact \ref{fact:ge7} (i), given $i,j\in [m]$ and a bipartite graph on $(V_i, V_j)$ containing a matching of size three, at most five vertices in $U$ can share this link graph on $(V_i, V_j)$. Since there are $2^{16}$ (labeled) bipartite graphs on $(V_i, V_j)$, we get that $|\mathcal{T}_{\ge 7}^3|\le \binom m2 \cdot 2^{16}\cdot 5$.

\smallskip

To see Part (ii), let $\D$ denote the digraph on $[m]$ such that $(i,j)\in E(\D)$ if and only if at least eight vertices $u_1,\dots, u_8\in U$ share the same link graph on $(V_i, V_j)$ such that $u_1 ij\in \T_{\ge 7}^2$, and the vertex cover is in $V_i$. We claim that $\deg_{\D}^-(i)\le 1$ for every $i\in [m]$ and consequently $e(\D)\le m$. Suppose instead,  there are $i,j,k\in [m]$ such that $(j, i), (k, i)\in E(\D)$, then eight vertices of $U$ share the same link graph on $(V_i, V_j)$, and (not necessarily different) eight vertices of $U$ share the same link graph on $(V_i, V_k)$. Thus we can pick four distinct vertices for each of $(V_i, V_j)$ and $(V_i, V_k)$ and obtain a structure forbidden by Fact \ref{fact:ge7} (ii), a contradiction. Note that there are $2\cdot \binom 42\cdot 8+2\cdot \binom 42=108$ (labeled) bipartite graphs on $(V_i, V_j)$ with at least seven edges and a vertex cover of two vertices both from $V_i$ or $V_j$. Furthermore, fixing one of these bipartite graphs, if $(i,j)\notin E(\D)$ and $(j,i)\notin E(\D)$, then, by the definition of $\D$, at most seven vertices in $U$ share this link graph. Hence
\[
|\T_{\ge 7}^2|\le \binom m2\cdot 108\cdot 7 + m |U|=756\binom m2  + m |U|.
\]

\smallskip

To see Part (iii), recall that $\A_i$ is the set of all edges of $\h$ with exactly $i$ vertices in $V'$. Then
\[
|\A_2|\le 6|\T_{\le 6}| + 7 |\T_{7}^1| + 8 |\T_{\ge 7}^2|+16|\T_{\ge 7}^3| + \binom 42 m |U|.
\]
Together with $|\T_{\le 6}|+|\T_{7}^1|+|\T_{\ge 7}^2|+|\T_{\ge 7}^3|=\binom m2 |U|$, we get,
\begin{align}
|\A_2|&\le7 \binom{m}2 |U| - |\mathcal T_{\le 6}| + |\T_{\ge 7}^2|+9|\T_{\ge 7}^3| + 6 m |U| \nonumber \\
         &\le 7 \binom{m}2 |U| - |\mathcal T_{\le 6}| + \binom m2\cdot (2^{16}\cdot 45+756)+ 7 m |U|\quad\text{by Parts (i), (ii)} \nonumber \\
         &< 7 \binom{m}2 |U| - |\mathcal T_{\le 6}| + 2^{22}\binom m2+ 7 m |U|. \label{eq:A2}
\end{align}

We know that $\sum_{u\in U}\deg(u) = 3|\A_0| + 2|\A_1| +|\A_2|$. Thus, by $|\A_0| \le \frac13\binom{|U|}2$, Claim \ref{clm:A01} and \eqref{eq:A2}, we have
\begin{align}
\sum_{u\in U}\deg(u) &\le \binom{|U|}2 +2m\binom{|U|}2+24m|U| + 7 \binom{m}2 |U| - |\mathcal T_{\le 6}|+ 2^{22}\binom m2+ 7 m |U|  \nonumber\\
&= \binom{|U|}2 +m|U|^2+30m|U| + 7 \binom{m}2 |U| - |\mathcal T_{\le 6}|+ 2^{22}\binom m2  \nonumber\\
&< \frac7{16}\binom{|U|}2|U|+\frac74 m|U|^2+ \frac7{16} \binom{4m}2|U| - |\mathcal T_{\le 6}|+ 2^{22}\binom m2 \quad\text{as } |U|>40  \nonumber\\
&= \frac {7}{16}\binom{n}2 |U| - |\mathcal T_{\le 6}| + 2^{22}\binom m2, \label{eq:sumU}
\end{align}
where the last inequality is due to $\binom{|U|}2 +4m|U| + \binom {4m}2=\binom{|U|+4m}2=\binom n2$.

On the other hand,  $\delta_1(\h)\ge (\frac7{16} - \r) \binom n2$ implies that $\sum_{u\in U}\deg(u) \ge (\frac7{16} - \r) \binom n2|U|$. Together with \eqref{eq:sumU}, this gives the desired bound for $|\T_{\le 6}|$.

\smallskip

To see Part (iv), note that \eqref{eq:Um} implies that $|U|<\frac34n<\frac34 16m=12m$. By Parts (i)--(iii), we have
\begin{align*}
|\mathcal T_{7}^1|&\ge \binom{m}2|U| -\left(\binom m2\cdot (2^{16}\cdot 5+756)+ m|U| \right) - \r \binom n2|U| - 2^{22}\binom m2 &&\\
&\ge  \binom{m}2|U| -  \r \binom n2|U| - 2^{23}\binom m2  \quad\text{as }|U|<12m \\
&\ge \binom{m}2|U| -  \r \binom n2|U| - 2^{19}\binom n2  \quad\text{as }m<\frac n4 \\
&> \binom{m}2|U| -  \r n^2|U| \quad \text{as } |U|>2^{19}/\r. \qedhere
\end{align*}
\end{proof}

For a triple $uij\in \T_{7}^1$, we call $v_1\in V_i$ and $v_2\in V_j$ \emph{centers} for $u$ if $\{v_1, v_2\}$ is the vertex cover of $L_{i,j}(u)$. Define $G$ as the graph on the vertex set $V'$ such that two vertices $v_1,  v_2\in V'$ are adjacent if and only if there are at least 16 vertices $u\in U$ such that $v_1, v_2$ are centers for $u$.

\begin{fact}\label{fact:C}
For every $i\in [m]$, at most one vertex $v\in V_i$ satisfies $\deg_G(v)>0$.
\end{fact}

\begin{proof}
Suppose to the contrary, some $V_i=\{a,b,c,d\}$ satisfies $\deg_G(a), \deg_G(b)>0$. Let $a'\in N_G(a), b'\in N_G(b)$ and assume that $a'\in V_j, b'\in V_k$ (it is possible to have $a'=b'$).
Pick $x\in V_j\setminus \{a', b'\}$ and $y\in V_k\setminus \{a', b', x\}$. By the definition of $G$, we can find $u_1,\dots, u_4, u_1', \dots, u_4'\in U$ such that $a, a'$ are centers for $u_l$ and $b, b'$ are centers for $u_l'$ for $l=1,\dots, 4$. This gives three copies of $\Y$ on $a x u_1u_2, a'cu_3u_4, b y u_1'u_2'$ immediately. 
So if $j=k$, then replacing $\Y_i, \Y_j$ by them in $\mathscr Y$ gives a larger $\Y$-tiling, a contradiction.
If $j\neq k$, then we have $a'\neq b'$, and we get one more copy of $\Y$ on $b'du_3'u_4'$. 
Replacing $\Y_i, \Y_j, \Y_k$ by these four copies of $\Y$ in $\mathscr Y$ gives a larger $\Y$-tiling, a contradiction.
\end{proof}

Let $C$ be the set of vertices $v\in V'$ such that $\deg_G(v)\ge 7$ and $\deg_G(v')\ge 2$ for some $v'\in N_G(v)$, where $N_G(v)$ denotes the neighborhood of $v$ in $G$.

\begin{claim}\label{clm:C}
$(1-2^{11}\r)m\le |C|\le m$.
\end{claim}

\begin{proof}
The upper bound follows from Fact \ref{fact:C} immediately.

To see the lower bound, we first show that
\begin{equation}
\label{eq:eG}
e(G)\ge (1-2^{10}\r)\binom m2.
\end{equation}
To see this, let $M$ be the set of pairs $i,j\in \binom{[m]}2$ such that there are at most $240$ vertices $u\in U$ satisfying that $uij\in \T_7^1$. By Claim \ref{clm:T} (iv), the number of triples $uij \not\in \T_{7}^1$ ($u\in U$, $i \ne j \in [m]$) is at most $\r n^2 |U|$. Thus
\[
|M|\le \frac{\r n^2 |U|}{|U|-240}\le \frac{\r n^2 |U|}{\frac23|U|}= \frac{3\r n^2}2<\frac{3\r (16m)^2}2<2^{10}\r \binom m2.
\]
where the second last inequality follows from \eqref{eq:Um}.
Fix a pair $i,j\in \binom {[m]}2\setminus M$. There are at least $241=16\cdot 15+1$ vertices $u\in U$ satisfying that $uij\in \T_7^1$. Since $V_i\times V_j$ contains 16 pairs of vertices, by the pigeonhole principle, some pair of vertices $v_1\in V_i,  v_2\in V_j$ are centers for at least 16 vertices $u\in U$, namely, $v_1v_2\in G$. Thus \eqref{eq:eG} follows.

By Fact \ref{fact:C}, there are at most $m$ vertices with positive degree in $G$. For convenience, define $V''\subset V'$ as an arbitrary set of $m$ vertices that contains all the vertices with positive degree in $G$. Furthermore, for any integer $t<m$, let $D_t\subseteq V''$ denote the set of vertices $v$ such that $\deg_G(v)\le t$. Let $F\subseteq (V''\setminus D_1)$ denote the set of vertices $v$ such that $N_G(v)\subseteq D_1$. We have
\[
2e(G)\le t|D_t|+(m-1)(m -|D_t|) = m(m-1)-(m-t-1)|D_t|.
\]
Together with \eqref{eq:eG}, it gives $|D_t|\le 2^{10}\r\frac{m(m-1)}{m-t-1}$. By definition, each vertex $v\in F$ satisfies $\deg_G(v)\ge 2$, and its neighborhood is contained in $D_1$ (thus the vertices in $F$ have disjoint neighborhoods). This implies that $|F|\le |D_1|/2$. Recall that $C= V''\setminus (D_6\cup F)$. Since $D_6$ and $D'_2$ are not necessarily disjoint,
\[
|C|\ge m-|D_6|-|F|\ge m-2^{10}\r\frac{m(m-1)}{m-7}- 2^{10}\r\frac{m(m-1)}{2(m-2)}\ge (1-2^{11}\r)m.
\]
as claimed.
\end{proof}

Let $I_C$ be the set of all $i\in [m]$ such that $V_i\cap C\neq \emptyset$. Fact \ref{fact:C} and Claim~\ref{clm:C} together imply that $| I_C|=|C|\ge (1-2^{11}\r)m$. Let $A = (\bigcup_{i\in I_C}V_i\setminus C) \cup U$.

\begin{claim}\label{clm:edge}
$\h[A]$ contains no copy of $\Y$, thus $e(A)\le \frac13\binom n2$.
\end{claim}

\begin{proof}
The first half of the claim implies the second half by Fact \ref{fact:Y}. Suppose instead, $\h[A]$ contains a copy of $\Y$, denoted by $\Y_0$, on $V_0$. Since $\h[U]$ contains no copy of $\Y$, $V_0$ must intersect some $V_i$ with $i\in I_C$.  Without loss of generality, suppose that $V_1, \dots, V_{j}$ contain the vertices of $V_0\setminus U$ for some $1\le j\le 4$ (recall that $V_i=V(\Y_i)$ for $i\in [m]$). Here we separate two cases.

\smallskip
\noindent \textbf{Case 1.} For any $i\in [j]$, $|V_i\cap V_0|\le 2$.
\smallskip

For $i \in [j]$, let $\{c_i\} =V_i\cap C$, and suppose that $d_i\in V_i\setminus (V_0\cup \{c_i\})$. For each $i\in [j]$, since $\deg_G(c_i)\ge 7$, we can pick distinct $v_i \in N_G(c_i)\setminus (V_1\cup \dots \cup V_j)$. By Fact~\ref{fact:C}, $v_1, \dots, v_j$ are contained in different members of $\mathscr Y$ (also different from $\Y_1, \dots, \Y_j$). Let
$v'_1, \dots, v'_j$ be arbitrary vertices in these members of $\mathscr Y$, respectively, which are different from $v_1, \dots, v_j$. For every $i\in [j]$, since $c_i, v_i$ are centers for at least $16$ vertices of $U$, we find a set of four vertices $u_i^1,\dots, u_i^4\in U\setminus V_0$ disjoint from the previous ones such that $c_i, v_i$ are centers for them. This is possible because $|V_0\cap U|\le 4-j$ and the number of available vertices in $U$ is thus at least $16-(4-j)=12+j\ge 4j$.

Note that for $i \in [j]$, $c_i v'_i u_i^1 u_i^2, d_i v_i u_i^3 u_i^4$ span two copies of $\Y$. Together with $\Y_0$, this gives $2j+1$ copies of $\Y$ while using vertices from $2j$ members of $\mathscr Y$, contradicting the maximality of $\mathscr Y$.

\smallskip
\noindent \textbf{Case 2.} There exists $i_0\in [j]$, such that $|V_{i_0}\cap V_0|=3$.
\smallskip

Note that $j=1$ or 2 in this case. Without loss of generality, assume that $|V_1\cap V_0|=3$. First assume that $j=1$ (then $|V_0\cap U|=1$). Let $\{c_1\} =V_1\cap C$. By the definition of $C$, there exists $c_2\in N_G(c_1)$ such that  $\deg_G(c_2)\ge 2$. Let $c_3\ne c_1$ be a neighbor of $c_2$ in $G$. Assume that $\Y_{i_2}, \Y_{i_3}\in \mathscr Y$ contains $c_2, c_3$, respectively.
By the definition of $G$, we can find $u_1, \dots, u_6 \in U\setminus V_0$ such that $c_1, c_2$ are centers for $u_1, u_2$, and $c_2, c_3$ are centers for $u_3, u_4, u_5, u_6$. Thus, $c_1 w_1u_1u_2$, $c_2 w_3 u_3 u_4$, $c_3 w_2 u_5 u_6$ span three copies of $\Y$, where $w_1, w_2$ are two vertices in $V_{i_2}\setminus \{c_2\}$ and $w_3\in V_{i_3}\setminus \{c_3\}$. Together with $\Y_0$, it gives four copies of $\Y$ while using vertices from three members of $\mathscr Y$, contradicting the maximality of $\mathscr Y$.

Now assume that $j=2$, that is, $|V_0\cap V_2|=1$. We pick $c_2, c_3, u_1, \dots, u_6$ in the same way as in the $j=1$ case. If $c_2\in V_2$, then this gives four copies of $\Y$ by using vertices from three members of $\mathscr Y$, a contradiction. Otherwise, let $\{c_4\} =V_2\cap C$ and pick $c_5\in N_G(c_4)\setminus \{c_1, c_2, c_3\}$ (this is possible because $\deg_G(c_4)\ge 7$). Suppose that $\Y_{i_5}$ contains $c_5$. We pick four new vertices $u_7, \dots, u_{10}\in U$ for whom $c_4, c_5$ are centers. Thus, we can form two copies of $\Y$ by using vertices from $\Y_2, \Y_{i_5}$ and $u_7, \dots, u_{10}$. Together with the four copies of $\Y$ given in the previous case, we obtain six copies of $\Y$ while using vertices from five members of $\mathscr Y$, a contradiction.
\end{proof}

Note that the edges not incident to $C$ are either contained in $A$ or incident to some $V_i$, $i\notin I_C$. By Claim \ref{clm:edge}, $C$ is incident to all but at most
\begin{align*}
& e(A) + 4\cdot 2^{11}\r m \binom {n-1}2  < \frac13 \binom n2  + 2^{10}\r (4m) n^2 \\
& <2^{10}\r n^2\left(\frac1{2^{10}\r}+4m \right)<2^{10}\r n^3,
\end{align*}
edges, where the last inequality holds because $|U|>\frac 1{2^{10}\r}$. Since $|C|\le m\le n/4$, we can pick a set $B\subseteq V\setminus C$ of order $\lfloor \frac34n\rfloor$. Then $e(B)<2^{10}\r n^3$, which implies that $\h$ is $2^{10}\r$-extremal.
\end{proof}

In Claim~\ref{clm:edge} we proved that $\h[A]$ contains no copy of $\Y$, where, by Claim~\ref{clm:C},
\[
|A|=n - m - 3(m-|C|)\ge n-\frac n4 - 3\cdot 2^{11}\r m\ge (1-2^{11}\r)\frac34n.
\]
We summarize this in a lemma and will use it in our forthcoming paper \cite{HZ3}.
\begin{lemma}\label{lem:Y1}
For any $\r >0$, there exists an integer $n_0$ such that the following holds. Suppose $\h$ is a 3-graph on $n>n_0$ vertices with
\[
\delta_1(\h)\ge \left(\frac7{16} - \r \right)\binom n {2},
\]
then there is a $\Y$-tiling covering all but at most $2^{19}/\r$ vertices of $\h$ unless $\h$ contains a set of order at least $(1-2^{11}\r)\frac34 n$ that contains no copy of $\Y$.
\end{lemma}

\section{The Extremal Theorem}
In this section we prove Theorem \ref{lemE}. Let $n$ be sufficiently large and $\h$ be a 3-graph on $n$ vertices satisfying \eqref{eqdeg}. Assume that $\h$ is $\beta$-extremal, namely, there is a set $B\subseteq V(\h)$, such that $|B| = {\lfloor\frac34 n\rfloor}$ and $e(B)\le \beta n^3$. For the convenience of later calculations, we let $\e_0 = 18\beta$ and derive that
\begin{equation}\label{eqB}
e(B)< \e_0 \binom {|B|}3.
\end{equation}

Let us outline our proof here. We define two disjoint sets $A', B' \subseteq V(\h)$ such that $A'$ consists of the vertices with high degree in $B$ and $B'$ consists of the vertices with low degree in $B$.
We will show that $A'\approx A$ and $B'\approx B$ (Claim \ref{clm:size}). To illustrate our proof ideas,
suppose that we are in an ideal case with $n\in 4\mathbb N$, $A'=A$, and $B'=B$. In this case we arbitrarily partition $B'$ into three sets $B_1, B_2$ and $B_3$ of equal size, and find a labelling $B_1=\{b_1,\dots, b_{n/4}\}$, $B_2=\{b'_1,\dots, b'_{n/4}\}$ and $B_3=\{b_1'',\dots, b''_{n/4}\}$ such that $\Gamma$ has large minimum degree, 
where $\Gamma$ denotes the bipartite graph on $(A, [n/4])$ in which $x i\in \Gamma$ for $x\in A$ and $i\in [n/4]$ if and only if $b_i b_i', b_i' b_i''\in N_{\h}(x)$. It is easy to find a Hamilton cycle in $\Gamma$, which gives rise to a loose Hamilton cycle in $\h$. 
In our actual proof, we first build a short loose path $P$ that covers all the vertices of $V_0: = V(\h)\setminus (A'\cup B')$ (and some vertices from $A'$ and $B'$) such that $|A'\setminus V(P)|/ |B'\setminus V(P)|$ is at least $1/3$ (Claim \ref{clm:path}). This is possible because of the minimum degree condition \eqref{eqdeg} and the fact that $V_0$ is small. We next extend $P$ to a loose path $Q$ such that $|A'\setminus V(Q)|/ |B'\setminus V(Q)|$ is about $1/3$ and finally find a loose Hamilton path on $V(\h)\setminus Q $ by following the approach for the ideal case  (Lemma \ref{lem:finish_3cyc}).

\subsection{Classifying vertices}
Let $\e_1=8\sqrt{\e_0}$ and $A= V(\h)\setminus B$. Assume that the partition $A$ and $B$ satisfies that $|B| = {\lfloor\frac34 n\rfloor}$ and \eqref{eqB}. In addition, assume that $e(B)$ is the smallest among all the partitions satisfying these conditions. We now define
\begin{align*}
&A':=\left\{ v\in V\mid \deg (v,B)\ge (1-\e_1)\binom{|B|}{2} \right\}, \\
&B':=\left\{ v\in V\mid\deg (v,B)\le \e_1\binom{|B|}{2} \right\}, \\
& V_0=V\setminus(A'\cup B').
\end{align*}

\begin{claim}\label{clm:eB}
$A\cap B'\neq \emptyset$ implies that $B\subseteq B'$, and $B\cap A'\neq \emptyset$ implies that $A\subseteq A'$.
\end{claim}

\begin{proof}
First, assume that $A\cap B'\neq \emptyset$. Then there is some $u\in A$ which satisfies that $\deg(u,B)\le \e_1\binom{|B|}2$. If there exists some $v\in B\setminus B'$, namely, $\deg(v, B)>\e_1\binom{|B|}2$, then we can switch $u$ and $v$ and form a new partition $A''\cup B''$ such that $|B''|=|B|$ and $e(B'')<e(B)$, which contradicts the minimality of $e(B)$.

Second, assume that $B\cap A'\neq \emptyset$. Then some $u\in B$ satisfies that $\deg(u,B)\ge (1-\e_1)\binom{|B|}2$. Similarly, by the minimality of $e(B)$, we get that for any vertex $v\in A$, $\deg(v, B)\ge (1-\e_1)\binom{|B|}2$, which implies that $A\subseteq A'$.
\end{proof}

\begin{claim}\label{clm:size}
$\{|A\setminus A'|, |B\setminus  B'|, |A'\setminus  A|, |B'\setminus  B|\}\le \frac{\e_1}{64}|B|$ and $|V_0|\le \frac{\e_1}{32}|B|$.
\end{claim}

\begin{proof}
First assume that $|B\setminus B'|> \frac{\e_1}{64}|B|$. By the definition of $B'$ and the assumption $\e_1=8\sqrt{\e_0}$, we get that
\[
e(B) > \frac 13 \e_1\binom {|B|}2 \cdot \frac{\e_1}{64}|B| > \frac{\e_1^2}{64}\binom {|B|}3  = \e_0 \binom {|B|}3,
\]
which contradicts \eqref{eqB}.

Second, assume that $|A\setminus A'|> \frac{\e_1}{64}|B|$. Then by the definition of $A'$, for any vertex $v\notin A'$, we have that $\overline \deg(v,B)> \e_1\binom{|B|}{2}$. So we get
\[
\overline e(ABB)> \frac{\e_1}{64}|B| \cdot \e_1 \binom{|B|}2=\e_0|B| \binom{|B|}2> 3\e_0 \binom{|B|}3.
\]
Together with \eqref{eqB}, this implies that
\begin{align*}
\sum_{b\in B}\overline\deg(b)&\ge 3\overline e(B)+2 \overline e(ABB) 
                              > 3(1 - \e_0) \binom{|B|}3 + 6\e_0 \binom{|B|}3 = 3(1+\e_0)\binom{|B|}3.
\end{align*}
By the pigeonhole principle, there exists $b\in B$, such that
\[
\overline\deg(b)  > (1+\e_0) \binom{|B|}2 =(1+\e_0)\binom{\lfloor \frac{3n}4\rfloor-1}2> \binom{\lfloor \frac{3n}4\rfloor}2,
\]
where the last inequality follows from the assumption that $n$ is large enough. This contradicts \eqref{eqdeg}.

Consequently,
\begin{align*}
&|A'\setminus A|=|A'\cap B|\le |B\setminus B'|\le \frac{\e_1}{64}|B|,   \\
&|B'\setminus B|=|A\cap B'|\le |A\setminus A'|\le \frac{\e_1}{64}|B|, \\
&|V_0|=|A\setminus A'|+|B\setminus B'|\le \frac{\e_1}{64}|B|+\frac{\e_1}{64}|B|=\frac{\e_1}{32}|B|. \qedhere
\end{align*}
\end{proof}

We next show that we can connect any two vertices of $B'$ with a loose path of length two without using any fixed $\frac n8$ vertices of $V$.

\begin{claim}\label{clm:conn}
For every pair of vertices $u, v\in B'$ and every vertex set $S\subseteq V$ with $|S|\le n/8$, there exist $a\in A'\setminus S$ and $b_1,b_2\in B'\setminus S$ such that $ub_1a, ab_2v\in E(\h)$.
\end{claim}

\begin{proof}
For any $x\in B'$, by \eqref{eqdeg}, we have that $\overline \deg(x) \le \binom{\lfloor\frac34 n\rfloor}2=\binom{|B|}2$. So by the definition of $B'$,
\[
\overline\deg(x, AB)\le \overline\deg(x) - \overline\deg(x, B) \le \binom {|B|}2- (1-\e_1)\binom {|B|}2=\e_1 \binom {|B|}2.
\]
By Claim \ref{clm:size}, we get that
\begin{align}
\overline\deg(x, A'B')&\le \overline\deg(x, AB)+|A'\setminus A|\cdot |B'|+|B'\setminus B|\cdot |A'|  \nonumber \\
&\le \e_1 \binom {|B|}2+ \frac{\e_1}{64}|B|n\le 2\e_1\binom {|B|}2. \label{eq:barAB}
\end{align}
Consider a bipartite graph $G$ on $A\setminus S$ and $B\setminus S$ with pairs $ab\in E(G)$ if and only if $uab,vab\in E(\h)$. Since $|S|\le \frac n8$, we have $|A\setminus S|\ge \frac {|A|}2\ge \frac{|B|}6$ and $|B\setminus S|> \frac{|B|}2$, so $|A\setminus S|\cdot |B\setminus S|> \frac16 \binom {|B|}2>8\e_1\binom{|B|}2$. Consequently,
\[
e(G)\ge |A\setminus S|\cdot |B\setminus S|-4\e_1\binom{|B|}2 \ge \frac12|A\setminus S|\cdot |B\setminus S|>|A\setminus S|.
\] 
Hence there exists a vertex $a\in A\setminus S$ such that $\deg_G(a)\ge 2$. By picking $b_1, b_2\in N_G(a)$ we 
finish the proof.
\end{proof}

\subsection{Building a short path}

\begin{claim}\label{clm:Bpath}
Suppose that $|A\cap B'|=q>0$. Then there exists a family $\mathcal P_1$ of vertex disjoint loose paths in $B'$, where
\[
\mathcal P_1 \text{ consists of  }\begin{cases}
\text{one edge} &\text {if } q=1 \text{ and }n\notin 4\mathbb N\\
\text{two edges }e_1, e_2 \text{ with } |e_1\cap e_2|\le 1 &\text{if }q=1 \text{ and } n\in 4\mathbb N\\
2q \text{ disjoint edges} &\text{if } q\ge 2
\end{cases}
\]
\end{claim}

\begin{proof}
Let $|A\cap B'|=q>0$. Since $A\cap B'\neq \emptyset$, by Claim \ref{clm:eB}, we get $B\subseteq B'$, which implies $|B'|=\lfloor\frac34 n\rfloor+q$.

By Claim \ref{clm:size}, 
we get that $q=|A\cap B'|\le |A\setminus A'|\le \frac{\e_1}{64}|B|$.
Hence for any vertex $b$ in $B'$,
\begin{align} \label{eq:analeft}
\deg\left(b, B' \right)&\le \deg\left(b,B \right)+|B'\setminus B|(|B'|-1)  \nonumber    \\
                                              & \le \e_1 \binom {|B|}2 + q (|B'| - 1)< 2\e_1 \binom {|B|}2.
\end{align}

Now we assume that $q=1$, so $|B'|-1=\lfloor\frac34 n \rfloor$. By \eqref{eqdeg}, for any $b\in B'$,
\begin{align*}
\deg \left(b, B'\right) &\ge \binom{n-1}2 - \binom{\lfloor\frac34 n \rfloor}2 + c - \left[\binom{n-1}2 - \binom{|B'|-1}2  \right] =c,
\end{align*}
where $c=1$ if $n\notin 4\mathbb N$ and $c=2$ otherwise. The $n\notin 4\mathbb N$ case is trivial since
$B'$ actually contains at least $|B'|/3>1$ edges. If $n\in 4\mathbb N$, then we have $\deg (b, B' )\ge 2$. Assume that $B'$ does not contain the desired structure. Then any two distinct edges of $B'$ share exactly two vertices. Fix an edge $e_0= v_1v_2v_3$ of $B'$ and two vertices $u, u' \in B'\setminus e_0$. Then every edge of $B'$ containing $u$ must have its two other vertices in $e_0$. Since $\deg (u, B' )\ge 2$, the link graph of $u$ contains at least two pairs of vertices of $e_0$. So does the link graph of $u'$. We thus find a loose path of length two from $u$ to $u'$ because two distinct pairs on $e_0$ share exactly one vertex.

Second, assume that $q>1$. In this case we construct $2q$ disjoint edges  greedily. By \eqref{eqdeg} and $|B'|=\lfloor\frac34 n\rfloor+q$, for any $b\in B'$,
\begin{align*}
\deg \left(b, B' \right) &\ge \binom{n-1}2 - \binom{\lfloor\frac34 n \rfloor}2 + c - \left[\binom{n-1}2 - \binom{|B'|-1}2  \right] \\
                                                &>\binom{|B'|-1}2 - \binom{\lfloor\frac34 n \rfloor}2\\
                                                &\ge (q-1)\left\lfloor\frac34 n \right\rfloor,
\end{align*}
which implies that $e(B')> \frac13 |B'| (q-1)\lfloor\frac34 n\rfloor$. Suppose we have found $i<2q$ disjoint edges of $B'$. By \eqref{eq:analeft}, there are at most $3(2q - 1)\cdot 2\e_1 \binom {|B|}2$ edges of $B'$ intersecting these $i$ edges. Hence, there are at least
\begin{align*}
e(B')-3(2q - 1)\cdot 2\e_1 \binom {|B|}2 &\ge \frac13 |B'| (q-1) \left\lfloor\frac34 n \right\rfloor -   6(2q - 1)\e_1 \binom {|B|}2\\
&\ge \frac{2(q-1)}{3} \binom{|B|}2 -  6(2q - 1)\e_1 \binom {|B|}2\\
&=\frac23 \left[(q-1) -  9(2q - 1)\e_1\right] \binom {|B|}2
\end{align*}
edges not intersecting the existing $i$ edges. This quantity is positive provided that $\e_1<\frac{q-1}{9(2q-1)}$. Thus, $\e_1<\frac1{27}$ suffices since the minimum of $\frac{q-1}{9(2q-1)}$, $q>1$ is $\frac1{27}$ attained by $q=2$.
\end{proof}

\begin{remark}
Claim \ref{clm:Bpath} is the only place where the constant $c$ from \eqref{eqdeg} is used.
\end{remark}

The goal of this subsection is to prove the following claim.
\begin{claim}\label{clm:path} 
There exists a loose path $P$ in $\h$ with the following properties:
\begin{itemize}
\item $V_0\subseteq V(P)$,
\item $|V(P)|\le \frac{\e_1}{4} |B|$,
\item $|B'\setminus V(P)|\le 3|A'\setminus V(P)|-1$,
\item both ends of $P$ are in $B'$.
\end{itemize}
\end{claim}

\begin{proof}
We split into two cases here.

\smallskip
\noindent\textbf{Case 1.} $A\cap B'\neq \emptyset$.
\smallskip

By Claim \ref{clm:eB}, $A\cap B'\neq \emptyset$ implies that $B\subseteq B'$, which implies that $V_0\subseteq A$. Let $q=|A\cap B'|$. We first apply Claim \ref{clm:Bpath} and find a family $\mathcal P_1$ of vertex disjoint loose paths on at most $6q$ vertices of $B'$. Next we put each vertex of $V_0$ into a loose path of length two with four vertices from $B$ (so in $B'$) such that these paths are pairwise vertex disjoint and also vertex disjoint from the paths in $\mathcal P_1$. Let $V_0=\{x_1,\dots, x_{|V_0|}\}$. Suppose that we have found loose paths for $x_1, \dots, x_i$ with $i<|V_0|$. 
Since $A\setminus A' =V_0 \dot\cup (A\cap B')$, by Claim \ref{clm:size}, we have 
\begin{equation}
\label{eq:qV0}
q+|V_0| = |A\setminus A'|\le \frac{\e_1}{64}|B|.
\end{equation}
Thus,
\[
4i+6q< 4|V_0|+6q\le 6(|V_0|+q) \le \frac{3\e_1}{32} |B|
\] 
and consequently at most $\frac{3\e_1}{32} |B| (|B|-1) = \frac{3\e_1}{16}\binom{|B|}2$ pairs of $B$ intersect the existing paths. By the definition of $V_0$, $\deg(x_{i+1}, B)> \e_1\binom{|B|}2$. Since every graph on $n\ge 4$ vertices and $m\ge n$ edges contains two vertex disjoint edges, we can find two vertex disjoint pairs in the link graph of $x_{i+1}$ in $B$.

Denote by $\mathcal P_2$ the family of the loose paths that we obtained so far. Now we want to glue paths of $\mathcal P_2$ together to a single loose path. For this purpose, we apply Claim \ref{clm:conn} repeatedly to connect the ends of two loose paths while avoiding previously used vertices. This is possible because $|V(\mathcal P_2)|\le 5|V_0|+6q$ and at most $3(|V_0|+2q-1)$ vertices will be used to connect the paths in $\mathcal P_2$. By \eqref{eq:qV0}, the resulting loose path $P$ satisfies 
\[
|V(P)|\le 8|V_0|+12q -3 < 12 \cdot \frac{\e_1}{64} |B|< \frac{\e_1}4 |B|.
\]
We next show that $|B'\setminus V(P)|\le 3|A'\setminus V(P)|-1$. To prove this, we split into three cases according to the structure of $\mathcal P_1$. Note that $|B'|=\lfloor \frac{3n}4\rfloor+q$ and $|A'|=\lceil \frac n4 \rceil - |V_0|-q$.

First, assume that $q>1$. Our construction shows that $\mathcal P_1$ consists of $2q$ disjoint edges in $B'$. So $|V(P)\cap A'|=|V_0|+2q-1$ and $|V(P)\cap B'|=4|V_0|+3\cdot 2q+2(|V_0|+2q-1)=6|V_0|+10q-2$. Thus,
\begin{align*}
|B'\setminus V(P)|&= \left\lfloor\frac{3n}4 \right\rfloor+q- (6|V_0|+10q-2) \\
&\le 3\left(\left\lceil\frac n4 \right\rceil-2|V_0|-3q+1\right)-1 = 3|A'\setminus V(P)|-1.
\end{align*}

Second, assume that $q=1$ and $n\in 4\mathbb N$. Then $\mathcal P_1$ consists of a loose path of length two or two disjoint edges. For the first case, we have that $|V(P)\cap A'|=|V_0|$ and $|V(P)\cap B'|=4|V_0|+2|V_0|+5=6|V_0|+5$. Thus,
\[
|B'\setminus V(P)|= \frac{3n}4+1- (6|V_0|+5) =3\left(\frac n4-2|V_0|-1\right)-1 = 3|A'\setminus V(P)|-1.
\]
In the second case, we have that $|V(P)\cap A'|=|V_0|+1$ and $|V(P)\cap B'|=4|V_0|+2(|V_0|+1)+6=6|V_0|+8$. Thus,
\[
|B'\setminus V(P)|= \frac{3n}4+1- (6|V_0|+8) =3\left(\frac n4-2|V_0|-2\right)-1 = 3|A'\setminus V(P)|-1.
\]
Third, assume that $q=1$ and $n\notin 4\mathbb N$, so $\mathcal P_1$ contains only one edge. We have $|V(P)\cap A'|=|V_0|$ and $|V(P)\cap B'|=4|V_0|+2|V_0|+3=6|V_0|+3$. Let $n=4k+2$ with some $k\in \mathbb Z$, so $|A|=k+1$, $|B|=3k+1$, $|B'|=3k+2$ and $|A'|=k-|V_0|$. Thus,
\[
|B'\setminus V(P)|=3k+2- (6|V_0|+3) =3(k-2|V_0|)-1 = 3|A'\setminus V(P)|-1.
\]

\medskip

\noindent\textbf{Case 2.} $A\cap B'= \emptyset$.
\smallskip

Note that $A\cap B'= \emptyset$ means that $B'\subseteq B$. The difference from the first case is that we do not need to construct $\mathcal P_1$.

First we will put every vertex in $V_0$ into a loose path of length two together with four vertices from $B'$. By Claim \ref{clm:size}, $|B\setminus B'|\le \frac{\e_1}{64}|B|$ and thus for any vertex $x\in V_0$,
\begin{equation}
\deg(x, B') \ge \deg(x,B)-|B\setminus B'|\cdot (|B|-1) \ge \e_1\binom{|B|}2 - \frac{\e_1}{32}\binom{|B|}2.
\label{eq:medium}
\end{equation}
Similar as in Case 1, let $V_0=\{x_1,\dots, x_{|V_0|}\}$ and suppose that we have found loose paths for $x_1, \dots, x_i$ with $i<|V_0|$. By Claim \ref{clm:size}, $|V_0|\le \frac{\e_1}{32}|B|$. Thus, we have $4i< 4|V_0|\le\frac{\e_1}{8} |B|$
and consequently at most $\frac{\e_1}{8} |B| (|B'|-1) \le \frac{\e_1}{4} \binom{|B|}2$ pairs of $B'$ intersect the existing $i$ loose paths. Then by \eqref{eq:medium}, we may find two vertex disjoint pairs in the link graph of $x_{i+1}$ in $B'$.

As in Case 1, we connect the paths that we obtained to a single loose path by applying Claim \ref{clm:conn} repeatedly. The resulting loose path $P$ satisfies that 
\[ 
|V(P)|= 5|V_0|+3(|V_0|-1)< 8 \cdot \frac{\e_1}{32} |B| =\frac{\e_1}4 |B|. 
\]
We next show that $|B'\setminus V(P)|\le 3|A'\setminus V(P)|-1$. Note that $|V(P)\cap A'|=|V_0|-1$ and $|V(P)\cap B'|=4|V_0|+2(|V_0|-1)=6|V_0|-2$. Since $B'\subseteq B$, we have $|A'|\ge |A'\cap A|=\lceil\frac n4\rceil - |V_0|$. Thus,
\begin{align*}
|B'\setminus V(P)|&= |B'| - (6|V_0|-2) \le 3\left \lceil\frac n4 \right\rceil -6|V_0|+2 \\
&\le 3\left(|A'|+|V_0|-2|V_0|+1\right)-1 \\ 
&= 3(|A'|-|V(P)\cap A'|)-1= 3|A'\setminus V(P)|-1. \qedhere
\end{align*}
\end{proof}

\subsection{Completing a Hamilton cycle} 
Let $P$ be the loose path given by Claim~\ref{clm:path}.
Suppose that $|B'\setminus V(P)|= 3|A'\setminus V(P)|-l$ for some integer $l\ge 1$. 
Since $P$ is a loose path, $|V(P)|$ is odd. Since $V= A' \cup B' \cup V_0$ and $V_0\subset V(P)$, we have 
\begin{equation}
\label{eq:VP}
|V(P)| + |B'\setminus V(P)|+ |A'\setminus V(P)| = n.
\end{equation}
Since $n$ is even, it follows that $|B'\setminus V(P)|+ |A'\setminus V(P)|$ is odd, which implies that $l=3|A'\setminus V(P)|-|B'\setminus V(P)|$ is odd. 

If $l>1$, then we extend $P$ as follows. Starting from an end $u$ of $P$ (note that $u\in B'$), we add an edge by using one vertex from $A'$ and one from $B'$. This is guaranteed by Claim~\ref{clm:conn}, which actually provides a loose path starting from $u$. We repeat this $\frac{l-1}2$ times. The resulting loose path $P'$ satisfies $|B'\setminus V(P')|= 3|A'\setminus V(P')|-1$. We claim that $|V(P')|\le \frac{3\e_1}{4} |B|$ (thus Claim~\ref{clm:conn} can be applied repeatedly).  Indeed, by \eqref{eq:VP} and $|V(P)|\le \frac{\e_1}4|B|$,
\begin{align*}
l &= 3|A'\setminus V(P)|-|B'\setminus V(P)|=4|A'\setminus V(P)| - (n - |V(P)|)\\
& \le 4|A'|-n+ \frac{\e_1}4 |B|.
\end{align*}
Since $|A'|\le |A| + |B\setminus B'| \le \lceil\frac n4\rceil+\frac{\e_1}{64} |B|$ from Claim \ref{clm:size}, we have $l\le \frac{\e_1}2 |B|$. Since $|V(P')| = |V(P)| + l-1$, we derive that $|V(P')|\le \frac{3\e_1}{4} |B|$.

Finally, since both ends of $P'$ are vertices in $B'$, we extend $P'$ by one more $ABB$ edge from each end, respectively.
Denote the ends of the resulting path $Q$ be $x_0, x_1\in A'$. Let $A_1=(A'\setminus V(Q))\cup \{x_0, x_1\}$ and $B_1=B'\setminus V(Q)$. Note that we have $|B_1|= 3(|A_1|-1)$. By Claim \ref{clm:size}, we have $|B_1\setminus B|\le |B'\setminus B|\le \frac{\e_1}{64} |B|$. Furthermore,
\begin{equation}
\label{eq:B1}
|B_1|\ge |B'|- \frac{3\e_1}{4} |B|\ge |B| - \frac{\e_1}{64}|B| - \frac{3\e_1}{4} |B| - 2> (1-\e_1)|B|.
\end{equation}
For a vertex $v\in A_1$, since $\overline \deg(v, B) \le \e_1 \binom {|B|}2$, we have 
\begin{align*}
\overline \deg\left(v, B_1 \right)& \le \overline \deg(v, B)+ |B_1\setminus B|\cdot (|B_1|-1) \\
&\le \e_1\binom {|B|}2 + \frac{\e_1}{64} |B| \left(1+ \frac{\e_1}{64}\right) |B| \\
& < 2\e_1\binom{|B|}2<3\e_1\binom{|B_1|}2,
\end{align*}
where the last inequality follows from \eqref{eq:B1}.
In addition, \eqref{eq:barAB} and \eqref{eq:B1} imply that for any vertex $v\in B_1$,
\[
\overline \deg(v, A_1B_1)\le \overline \deg(v, A'B')\le 2\e_1 \binom{|B|}2< \e_1 |B|^2 <4\e_1 |A_1| |B_1|.
\] 
We finally complete the proof of Theorem \ref{lemE} by applying the following lemma with $X=A_1$, $Z=B_1$, and $\rho=4\e_1$.

\begin{lemma}\label{lem:finish_3cyc}
Let  $\rho>0$ be sufficiently small and $n$ be sufficiently large. 
Suppose that $\h$ be a 3-graph on $n$ vertices with $V(\h)=X\dot\cup Z$ such that $|Z|= 3(|X|-1)$. Further, assume that for every vertex $v\in X$, $\overline \deg(v, Z)\le \rho\binom{|Z|}2$ and for every vertex $v\in Z$, $\overline \deg (v, XZ)\le \rho |X| |Z|$. Then given any two vertices $x_0, x_1\in X$, there is a loose Hamilton path from $x_0$ to $x_1$.
\end{lemma}

To prove Lemma \ref{lem:finish_3cyc}, we follow the approach\footnote{We proved this lemma by the absorbing method in the previous version of this manuscript.} in the proof of \cite[Lemma 3.4]{CzMo} given by Czygrinow and Molla, who applied a result of K\"uhn and Osthus \cite{KuOs06_pse}.
A bipartite graph $G=(A, B, E)$ with $|A|=|B|=n$ is called $(d, \e)$-\emph{regular} if for any two subsets $A'\subseteq A$, $B'\subseteq B$ with $|A'|, |B'|\ge \e n$,
\[
(1-\e)d\le \frac{e(A', B')}{|A'| |B'|}\le (1+\e)d,
\]
and $G$ is called \emph{$(d, \e)$-super-regular} if in addition $(1-\e)d n \le \deg(v) \le (1+\e) d n$ for every $v\in A\cup B$.

\begin{lemma}\cite[Theorem 1.1]{KuOs06_pse}\label{lem:random}
For all positive constants $d, \nu_0, \eta\le 1$ there is a positive $\e=\e(d, \nu_0, \eta)$ and an integer $N_0$ such that the following holds for all $n\ge N_0$ and all $\nu\ge \nu_0$. Let $G=(A, B, E)$ be a $(d, \e)$-super-regular bipartite graph whose vertex classes both have size $n$ and let $F$ be a subgraph of $G$ with $|F|= \nu |E|$. Choose a perfect matching $M$ uniformly at random in $G$. Then with probability at least $1-e^{-\e n}$ we have
\[
(1-\eta) \nu n \le |M\cap E(F)| \le (1+\eta) \nu n.
\]
\end{lemma}

\medskip
\begin{proof}[Proof of Lemma \ref{lem:finish_3cyc}]
Let $\e = \e(1, 7/8, 1/8)$ be given by Lemma \ref{lem:random} and $\rho=(\e/2)^4$.
Suppose that $n$ is sufficiently large and $\h$ is a 3-graph satisfying the assumption of the lemma.
Let $G$ be the graph of all pairs $uv$ in $Z$ such that $\deg(uv, X)\ge (1-\sqrt{\rho})|X|$. We claim that for any vertex $v\in Z$,
\begin{equation}\label{eq:dGY1}
\overline{\deg}_G(v)\le \sqrt{\rho}|Z|.
\end{equation}
Otherwise, some vertex $v\in Z$ satisfies $\overline \deg_G (v)> \sqrt{\rho}|Z|$. As each $u\notin N_G(v)$ satisfies $\overline{\deg}_{\h}(uv, X)> \sqrt{\rho}|X|$, we have
\[
\overline \deg_{\h}(v, XZ)> \sqrt{\rho}|Z| \cdot \sqrt{\rho} |X| = {\rho}|Z||X|,
\]
contradicting our assumption.

Let $m=|X|-1$.
Arbitrarily partition $Z$ into three sets $Z_1, Z_2, Z_3$, each of order $m$. By \eqref{eq:dGY1} and $|Z|= 3m$, we have
$\delta(G[Z_1, Z_2]), \delta(G[Z_2, Z_3])\ge (1-3\sqrt\rho)m$. It is easy to see that both $G[Z_1, Z_2]$ and $G[Z_2, Z_3]$ are $(1, \e)$-super-regular as $\e= 2\sqrt[4]\rho$.
For any $x\in X$, let $F_x^1:= \{ z z'\in E(G[Z_1, Z_2]): x z z'\in E(\h) \}$ and let $F_x^2:= \{ z z'\in E(G[Z_2, Z_3]): x z z'\in E(\h) \}$.
Since $\overline \deg(x, Z)\le \rho\binom{|Z|}2\le 5\rho m^2$, we have $|F_x^1|, |F_x^2|\ge (1-3\sqrt\rho) m^2 - 5\rho m^2 \ge \frac{7}{8} m^2$.
Let $M_1$ and $M_2$ be perfect matchings chosen uniformly at random from $G[Z_1, Z_2]$ and $G[Z_2, Z_3]$, respectively.
By applying Lemma \ref{lem:random} with $\nu_0 =7/8$ and $\eta=1/8$, for any $x\in X$, with probability at least $1-e^{-\e m}$, we have
\begin{equation}\label{eq:m12}
|M_1\cap E(F_x^1)|, |M_2\cap E(F_x^2)| \ge (1-\eta) \nu_0 m \ge \frac{49}{64} m.
\end{equation}
Thus there exist a matching $M_1$ in $G[Z_1, Z_2]$ and a matching $M_2$ in $G[Z_2, Z_3]$ such that \eqref{eq:m12} holds for all $x\in X$.
Label $Z_1=\{a_1,\dots, a_m\}$, $Z_2=\{b_1,\dots, b_{m}\}$ and $Z_3=\{c_1,\dots, c_{m}\}$ such that $M_1=\{a_1 b_1, \dots, a_{m} b_{m}\}$ and $M_2=\{b_1 c_1, \dots, b_{m} c_{m}\}$.
Let $\Gamma$ be a bipartite graph on $(X, [m])$ such that $x i\in E(\Gamma)$ if and only if $x a_i b_i, x b_i c_i\in E(\h)$ for $x\in X$ and $i\in [m]$. For every $i\in [m]$, since $a_i b_i, b_i c_i\in E(G)$, we have $\deg_\Gamma(i) \ge (1-2\sqrt\rho)|X|$ by the definition of $G$. On the other hand, by \eqref{eq:m12}, we have $\deg_\Gamma(x) \ge (1- 2(1- \frac{49}{64})) m = \frac{34}{64} m$ for any $x\in X$. By a result of Moon and Moser \cite{MoMo}, $\Gamma$ contains a Hamilton path $x_1 j_1\, x_2\, j_2\, \dots x_m\, j_m\, x_0$, where $[m]=\{j_1,\dots, j_m\}$ and $X=\{x_0, x_1,\dots, x_m\}$. 
Since for each $i\in [m]$, $x_i j_i, x_{i+1} j_i\in E(\Gamma)$ implies that $x_i a_{j_i} b_{j_i}, b_{j_i} c_{j_i} x_{i+1}\in E(\h)$ (with $x_{m+1}=x_0$), we get a loose Hamilton path of $\h$:
\[
x_1\,a_{j_1}\,b_{j_1}\,c_{j_1}\,x_2\,a_{j_2}\,b_{j_2}\,c_{j_2}\cdots x_{m}\, a_{j_m}\, b_{j_m}\, c_{j_m} x_0. \qedhere
\]
\end{proof}

\section{Concluding remarks}
\label{sec:CR}

Let $h^l_d(k, n)$ denote the minimum integer $m$ such that every $k$-uniform hypergraph $\h$ on $n$ vertices with minimum $d$-degree $\delta_d(\h) \ge m$ contains a Hamilton $l$-cycle (provided that $k- l$ divides $n$). In this paper we determine $h^1_1(3, n)$ for sufficiently large $n$. Can we apply the same approach to find other values $h^l_d(k, n)$? In the forthcoming paper \cite{HZ2}, we determine $h^l_{k-1}(k, n)$ for all $l<k/2$, improving the asymptotic results in \cite{KO,HS,KKMO}.

The authors of \cite{BHS} conjectured that $h^1_1(k, n)$ is asymptotically attained by a similar construction as the one supported $h_1^1(3,n)$. At present we cannot verify this conjecture because it seems that our success on $h^1_1(3, n)$ comes from the relation $d= k-2$ instead of the assumption $d=1$.

It was conjectured in \cite{RR} that $h^{k-1}_d(k, n)$ approximately equals to the minimum $d$-degree threshold for perfect matchings in $k$-graphs, in particular, $h^{k-1}_1 (k, n)= (1 - (1- 1/k)^{k-1} + o(1)) \binom{n}{k-1}$. This conjecture seems very hard because we do not even know the minimum $d$-degree threshold for perfect matchings in general.

The key lemma in our proof, Lemma~\ref{lem:Y}, shows that every 3-graph $\h$ on $n$ vertices with $\delta_1(\h)\ge  (7/16 - o(1)) \binom{n}{2}$ either contains an almost perfect $\Y$-tiling or is in the extremal case. Naturally this raises a question: what is the minimum vertex degree threshold for a perfect $\Y$-tiling? The corresponding codegree threshold was determined in \cite{KO} (asymptotically) and \cite{CDN} (exactly). We determine this minimum vertex degree threshold exactly in \cite{HZ3}
(Czygrinow \cite{Czy14} independently proved a similar result).

\section*{Acknowledgement}
We thank two referees for their valuable comments that improved the presentation of this paper.
\bibliographystyle{plain}
\bibliography{Jan2014}

\begin{thebibliography}{10}

\bibitem{AFHRRS}
N.~Alon, P.~Frankl, H.~Huang, V.~R{\"o}dl, A.~Ruci{\'n}ski, and B.~Sudakov.
\newblock Large matchings in uniform hypergraphs and the conjecture of {E}rd{\H
  o}s and {S}amuels.
\newblock {\em J. Combin. Theory Ser. A}, 119(6):1200--1215, 2012.

\bibitem{BHS}
E.~Bu{\ss}, H.~H{\`a}n, and M.~Schacht.
\newblock Minimum vertex degree conditions for loose {H}amilton cycles in
  3-uniform hypergraphs.
\newblock {\em J. Combin. Theory Ser. B}, 103(6):658--678, 2013.

\bibitem{Czy14}
A.~Czygrinow.
\newblock Minimum degree condition for {$C_4$}-tiling in 3-uniform hypergraphs.
\newblock {\em submitted}.

\bibitem{CDN}
A.~Czygrinow, L.~DeBiasio, and B.~Nagle.
\newblock Tiling 3-uniform hypergraphs with ${K}_4^3-2e$.
\newblock {\em Journal of Graph Theory}, 75(2):124--136, 2014.

\bibitem{CzKa}
A.~Czygrinow and V.~Kamat.
\newblock Tight co-degree condition for perfect matchings in 4-graphs.
\newblock {\em Electron. J. Combin.}, 19(2):Paper 20, 16, 2012.

\bibitem{CzMo}
A.~Czygrinow and T.~Molla.
\newblock Tight codegree condition for the existence of loose {H}amilton cycles
  in 3-graphs.
\newblock {\em SIAM J. Discrete Math.}, 28(1):67--76, 2014.

\bibitem{dirac}
G.~A. Dirac.
\newblock Some theorems on abstract graphs.
\newblock {\em Proc. London Math. Soc. (3)}, 2:69--81, 1952.

\bibitem{GPW}
R.~Glebov, Y.~Person, and W.~Weps.
\newblock On extremal hypergraphs for {H}amiltonian cycles.
\newblock {\em European J. Combin.}, 33(4):544--555, 2012.

\bibitem{HPS}
H.~H\`an, Y.~Person, and M.~Schacht.
\newblock On perfect matchings in uniform hypergraphs with large minimum vertex
  degree.
\newblock {\em SIAM J. Discrete Math}, 23:732--748, 2009.

\bibitem{HS}
H.~H\`an and M.~Schacht.
\newblock Dirac-type results for loose {Hamilton} cycles in uniform
  hypergraphs.
\newblock {\em Journal of Combinatorial Theory. Series B}, 100:332--346, 2010.

\bibitem{HZ2}
J.~Han and Y.~Zhao.
\newblock Minimum codegree threshold for hamilton $\ell$-cycles in k-uniform
  hypergraphs.
\newblock {\em Journal of Combinatorial Theory, Series A}, 132(0):194 -- 223,
  2015.

\bibitem{HZ3}
J.~Han and Y.~Zhao.
\newblock Minimum degree thresholds for ${C}_4^3$-tiling.
\newblock {\em Journal of Graph Theory, in press}, DOI: 10.1002/jgt.21833.

\bibitem{KK}
G.~Katona and H.~Kierstead.
\newblock Hamiltonian chains in hypergraphs.
\newblock {\em Journal of Graph Theory}, 30(2):205--212, 1999.

\bibitem{KKMO}
P.~Keevash, D.~K\"uhn, R.~Mycroft, and D.~Osthus.
\newblock Loose {Hamilton} cycles in hypergraphs.
\newblock {\em Discrete Mathematics}, 311(7):544--559, 2011.

\bibitem{Khan2}
I.~Khan.
\newblock Perfect matchings in 4-uniform hypergraphs.
\newblock {\em arXiv:1101.5675}.

\bibitem{Khan1}
I.~Khan.
\newblock Perfect matchings in 3-uniform hypergraphs with large vertex degree.
\newblock {\em SIAM J. Discrete Math.}, 27(2):1021--1039, 2013.

\bibitem{KMO}
D.~K\"uhn, R.~Mycroft, and D.~Osthus.
\newblock Hamilton $\ell$-cycles in uniform hypergraphs.
\newblock {\em Journal of Combinatorial Theory. Series A}, 117(7):910--927,
  2010.

\bibitem{KO}
D.~K\"uhn and D.~Osthus.
\newblock Loose {Hamilton} cycles in 3-uniform hypergraphs of high minimum
  degree.
\newblock {\em Journal of Combinatorial Theory. Series B}, 96(6):767--821,
  2006.

\bibitem{KuOs06_pse}
D.~K{\"u}hn and D.~Osthus.
\newblock Multicolored {H}amilton cycles and perfect matchings in pseudorandom
  graphs.
\newblock {\em SIAM J. Discrete Math.}, 20(2):273--286 (electronic), 2006.

\bibitem{KuOs14ICM}
D.~K{\"u}hn and D.~Osthus.
\newblock Hamilton cycles in graphs and hypergraphs: an extremal perspective.
\newblock {\em Proceedings of the International Congress of Mathematicians
  2014, Seoul, Korea}, Vol 4:381--406, 2014.

\bibitem{KOT}
D.~K{\"u}hn, D.~Osthus, and A.~Treglown.
\newblock Matchings in 3-uniform hypergraphs.
\newblock {\em J. Combin. Theory Ser. B}, 103(2):291--305, 2013.

\bibitem{MoMo}
J.~Moon and L.~Moser.
\newblock On Hamiltonian bipartite graphs.
\newblock {\em Israel Journal of Mathematics}, 1(3):163--165, 1963.

\bibitem{RR}
V.~R\"odl and A.~Ruci\'nski.
\newblock Dirac-type questions for hypergraphs — a survey (or more problems
  for Endre to solve).
\newblock {\em An Irregular Mind}, Bolyai Soc. Math. Studies 21:561--590, 2010.

\bibitem{RoRu14}
V.~R{\"o}dl and A.~Ruci{\'n}ski.
\newblock Families of triples with high minimum degree are {H}amiltonian.
\newblock {\em Discuss. Math. Graph Theory}, 34(2):361--381, 2014.

\bibitem{RRS06}
V.~R\"odl, A.~Ruci\'nski, and E.~Szemer\'edi.
\newblock A {D}irac-type theorem for 3-uniform hypergraphs.
\newblock {\em Combinatorics, Probability and Computing}, 15(1-2):229--251,
  2006.

\bibitem{RRS08}
V.~R\"odl, A.~Ruci\'nski, and E.~Szemer\'edi.
\newblock An approximate {D}irac-type theorem for k-uniform hypergraphs.
\newblock {\em Combinatorica}, 28(2):229--260, 2008.

\bibitem{RRS09}
V.~R{\"o}dl, A.~Ruci{\'n}ski, and E.~Szemer{\'e}di.
\newblock Perfect matchings in large uniform hypergraphs with large minimum
  collective degree.
\newblock {\em J. Combin. Theory Ser. A}, 116(3):613--636, 2009.

\bibitem{RRS11}
V.~R\"odl, A.~Ruci\'nski, and E.~Szemer\'edi.
\newblock Dirac-type conditions for {Hamiltonian} paths and cycles in 3-uniform
  hypergraphs.
\newblock {\em Advances in Mathematics}, 227(3):1225--1299, 2011.

\bibitem{Sze}
E.~Szemer{\'e}di.
\newblock Regular partitions of graphs.
\newblock In {\em Probl\`emes combinatoires et th\'eorie des graphes ({C}olloq.
  {I}nternat. {CNRS}, {U}niv. {O}rsay, {O}rsay, 1976)}, volume 260 of {\em
  Colloq. Internat. CNRS}, pages 399--401. CNRS, Paris, 1978.

\bibitem{TrZh12}
A.~Treglown and Y.~Zhao.
\newblock Exact minimum degree thresholds for perfect matchings in uniform
  hypergraphs.
\newblock {\em J. Combin. Theory Ser. A}, 119(7):1500--1522, 2012.

\bibitem{TrZh13}
A.~Treglown and Y.~Zhao.
\newblock Exact minimum degree thresholds for perfect matchings in uniform
  hypergraphs {II}.
\newblock {\em J. Combin. Theory Ser. A}, 120(7):1463--1482, 2013.

\end{thebibliography}

\end{document}